\documentclass{amsart}
\usepackage{amsfonts}
\usepackage{amsmath,amscd}
\usepackage{amsthm}
\usepackage{amssymb}
\usepackage{latexsym}

\numberwithin{equation}{section}

\usepackage{color}

\newtheorem{thm}{Theorem}[section]
\newtheorem{cor}[thm]{Corollary}
\newtheorem{lem}[thm]{Lemma}

\newtheorem{prop}[thm]{Proposition}

\newtheorem{example}[thm]{Example}
\newtheorem{defn}[thm]{Definition}

\newtheorem{rem}[thm]{Remark}

\numberwithin{equation}{section}
\begin{document}
\newcommand{\beqa}{\begin{eqnarray}}
\newcommand{\eeqa}{\end{eqnarray}}
\newcommand{\thmref}[1]{Theorem~\ref{#1}}
\newcommand{\secref}[1]{Sect.~\ref{#1}}
\newcommand{\lemref}[1]{Lemma~\ref{#1}}
\newcommand{\propref}[1]{Proposition~\ref{#1}}
\newcommand{\corref}[1]{Corollary~\ref{#1}}
\newcommand{\remref}[1]{Remark~\ref{#1}}
\newcommand{\er}[1]{(\ref{#1})}
\newcommand{\nc}{\newcommand}
\newcommand{\rnc}{\renewcommand}

\nc{\cal}{\mathcal}

\nc{\goth}{\mathfrak}
\rnc{\bold}{\mathbf}
\renewcommand{\frak}{\mathfrak}
\renewcommand{\Bbb}{\mathbb}

\nc{\trr}{\triangleright}
\nc{\trl}{\triangleleft}

\newcommand{\id}{\text{id}}
\nc{\Cal}{\mathcal}
\nc{\Xp}[1]{X^+(#1)}
\nc{\Xm}[1]{X^-(#1)}
\nc{\on}{\operatorname}
\nc{\ch}{\mbox{ch}}
\nc{\Z}{{\bold Z}}
\nc{\J}{{\mathcal J}}
\nc{\C}{{\bold C}}
\nc{\Q}{{\bold Q}}
\nc{\oC}{{\widetilde{C}}}
\nc{\oc}{{\tilde{c}}}
\nc{\ocI}{ \overline{\cal I}}
\nc{\og}{{\tilde{\gamma}}}
\nc{\lC}{{\overline{C}}}
\nc{\lc}{{\overline{c}}}
\nc{\Rt}{{\tilde{R}}}

\nc{\tW}{{\textsf{W}}}
\nc{\tG}{{\textsf{G}}}
\nc{\cY}{{\cal{Y}}}
\nc{\cZ}{{\cal{Z}}}
\nc{\tY}{{\textsf{Y}}}
\nc{\tZ}{{\textsf{Z}}}

\nc{\tz}{{\tilde{z}}}

\nc{\cL}{{\cal{L}}}
\nc{\cK}{{\cal{K}}}

\nc{\tw}{{\textsf{w}}}
\nc{\tg}{{\textsf{g}}}

\nc{\tx}{{\textsf{x}}}
\nc{\tho}{{\textsf{h}}}
\nc{\tk}{{\textsf{k}}}
\nc{\tep}{{\bf{\cal E}}}

\nc{\te}{{\textsf{e}}}
\nc{\tf}{{\textsf{f}}}
\nc{\tK}{{\textsf{K}}}

\nc{\odel}{{\overline{\delta}}}

\def\pr#1{\left(#1\right)_\infty}  

\renewcommand{\P}{{\mathcal P}}
\nc{\N}{{\Bbb N}}
\nc\beq{\begin{equation}}
\nc\enq{\end{equation}}
\nc\lan{\langle}
\nc\ran{\rangle}
\nc\bsl{\backslash}
\nc\mto{\mapsto}
\nc\lra{\leftrightarrow}
\nc\hra{\hookrightarrow}
\nc\sm{\smallmatrix}
\nc\esm{\endsmallmatrix}
\nc\sub{\subset}
\nc\ti{\tilde}
\nc\nl{\newline}
\nc\fra{\frac}
\nc\und{\underline}
\nc\ov{\overline}
\nc\ot{\otimes}

\nc\ochi{\overline{\chi}}
\nc\bbq{\bar{\bq}_l}
\nc\bcc{\thickfracwithdelims[]\thickness0}
\nc\ad{\text{\rm ad}}
\nc\Ad{\text{\rm Ad}}
\nc\Hom{\text{\rm Hom}}
\nc\End{\text{\rm End}}
\nc\Ind{\text{\rm Ind}}
\nc\Res{\text{\rm Res}}
\nc\Ker{\text{\rm Ker}}
\rnc\Im{\text{Im}}
\nc\sgn{\text{\rm sgn}}
\nc\tr{\text{\rm tr}}
\nc\Tr{\text{\rm Tr}}
\nc\supp{\text{\rm supp}}
\nc\card{\text{\rm card}}
\nc\bst{{}^\bigstar\!}
\nc\he{\heartsuit}
\nc\clu{\clubsuit}
\nc\spa{\spadesuit}
\nc\di{\diamond}
\nc\cW{\cal W}
\nc\cG{\cal G}
\nc\ocW{\overline{\cal W}}
\nc\ocZ{\overline{\cal Z}}
\nc\al{\alpha}
\nc\bet{\beta}
\nc\ga{\gamma}
\nc\de{\delta}
\nc\ep{\epsilon}
\nc\io{\iota}
\nc\om{\omega}
\nc\si{\sigma}
\rnc\th{\theta}
\nc\ka{\kappa}
\nc\la{\lambda}
\nc\ze{\zeta}

\nc\vp{\varpi}
\nc\vt{\vartheta}
\nc\vr{\varrho}

\nc\odelta{\overline{\delta}}
\nc\Ga{\Gamma}
\nc\De{\Delta}
\nc\Om{\Omega}
\nc\Si{\Sigma}
\nc\Th{\Theta}
\nc\La{\Lambda}

\nc\boa{\bold a}
\nc\bob{\bold b}
\nc\boc{\bold c}
\nc\bod{\bold d}
\nc\boe{\bold e}
\nc\bof{\bold f}
\nc\bog{\bold g}
\nc\boh{\bold h}
\nc\boi{\bold i}
\nc\boj{\bold j}
\nc\bok{\bold k}
\nc\bol{\bold l}
\nc\bom{\bold m}
\nc\bon{\bold n}
\nc\boo{\bold o}
\nc\bop{\bold p}
\nc\boq{\bold q}
\nc\bor{\bold r}
\nc\bos{\bold s}
\nc\bou{\bold u}
\nc\bov{\bold v}
\nc\bow{\bold w}
\nc\boz{\bold z}

\nc\ba{\bold A}
\nc\bb{\bold B}
\nc\bc{\bold C}
\nc\bd{\bold D}
\nc\be{\bold E}
\nc\bg{\bold G}
\nc\bh{\bold H}
\nc\bi{\bold I}
\nc\bj{\bold J}
\nc\bk{\bold K}
\nc\bl{\bold L}
\nc\bm{\bold M}
\nc\bn{\bold N}
\nc\bo{\bold O}
\nc\bp{\bold P}
\nc\bq{\bold Q}
\nc\br{\bold R}
\nc\bs{\bold S}
\nc\bt{\bold T}
\nc\bu{\bold U}
\nc\bv{\bold V}
\nc\bw{\bold W}
\nc\bz{\bold Z}
\nc\bx{\bold X}

\nc\ca{\mathcal A}
\nc\cb{\mathcal B}
\nc\cc{\mathcal C}
\nc\cd{\mathcal D}
\nc\ce{\mathcal E}
\nc\cf{\mathcal F}
\nc\cg{\mathcal G}
\rnc\ch{\mathcal H}
\nc\ci{\mathcal I}
\nc\cj{\mathcal J}
\nc\ck{\mathcal K}
\nc\cl{\mathcal L}
\nc\cm{\mathcal M}
\nc\cn{\mathcal N}
\nc\co{\mathcal O}
\nc\cp{\mathcal P}
\nc\cq{\mathcal Q}
\nc\car{\mathcal R}
\nc\cs{\mathcal S}
\nc\ct{\mathcal T}
\nc\cu{\mathcal U}
\nc\cv{\mathcal V}
\nc\cz{\mathcal Z}
\nc\cx{\mathcal X}
\nc\cy{\mathcal Y}

\nc\e[1]{E_{#1}}
\nc\ei[1]{E_{\delta - \alpha_{#1}}}
\nc\esi[1]{E_{s \delta - \alpha_{#1}}}
\nc\eri[1]{E_{r \delta - \alpha_{#1}}}
\nc\ed[2][]{E_{#1 \delta,#2}}
\nc\ekd[1]{E_{k \delta,#1}}
\nc\emd[1]{E_{m \delta,#1}}
\nc\erd[1]{E_{r \delta,#1}}

\nc\ef[1]{F_{#1}}
\nc\efi[1]{F_{\delta - \alpha_{#1}}}
\nc\efsi[1]{F_{s \delta - \alpha_{#1}}}
\nc\efri[1]{F_{r \delta - \alpha_{#1}}}
\nc\efd[2][]{F_{#1 \delta,#2}}
\nc\efkd[1]{F_{k \delta,#1}}
\nc\efmd[1]{F_{m \delta,#1}}
\nc\efrd[1]{F_{r \delta,#1}}

\nc\fa{\frak a}
\nc\fb{\frak b}
\nc\fc{\frak c}
\nc\fd{\frak d}
\nc\fe{\frak e}
\nc\ff{\frak f}
\nc\fg{\frak g}
\nc\fh{\frak h}
\nc\fj{\frak j}
\nc\fk{\frak k}
\nc\fl{\frak l}
\nc\fm{\frak m}
\nc\fn{\frak n}
\nc\fo{\frak o}
\nc\fp{\frak p}
\nc\fq{\frak q}
\nc\fr{\frak r}
\nc\fs{\frak s}
\nc\ft{\frak t}
\nc\fu{\frak u}
\nc\fv{\frak v}
\nc\fz{\frak z}
\nc\fx{\frak x}
\nc\fy{\frak y}

\nc\fA{\frak A}
\nc\fB{\frak B}
\nc\fC{\frak C}
\nc\fD{\frak D}
\nc\fE{\frak E}
\nc\fF{\frak F}
\nc\fG{\frak G}
\nc\fH{\frak H}
\nc\fJ{\frak J}
\nc\fK{\frak K}
\nc\fL{\frak L}
\nc\fM{\frak M}
\nc\fN{\frak N}
\nc\fO{\frak O}
\nc\fP{\frak P}
\nc\fQ{\frak Q}
\nc\fR{\frak R}
\nc\fS{\frak S}
\nc\fT{\frak T}
\nc\fU{\frak U}
\nc\fV{\frak V}
\nc\fZ{\frak Z}
\nc\fX{\frak X}
\nc\fY{\frak Y}
\nc\tfi{\ti{\Phi}}
\nc\bF{\bold F}
\rnc\bol{\bold 1}

\nc\ua{\bold U_\A}

\nc\qinti[1]{[#1]_i}
\nc\q[1]{[#1]_q}
\nc\xpm[2]{E_{#2 \delta \pm \alpha_#1}}  
\nc\xmp[2]{E_{#2 \delta \mp \alpha_#1}}
\nc\xp[2]{E_{#2 \delta + \alpha_{#1}}}
\nc\xm[2]{E_{#2 \delta - \alpha_{#1}}}
\nc\hik{\ed{k}{i}}
\nc\hjl{\ed{l}{j}}
\nc\qcoeff[3]{\left[ \begin{smallmatrix} {#1}& \\ {#2}& \end{smallmatrix}
\negthickspace \right]_{#3}}
\nc\qi{q}
\nc\qj{q}

\nc\ufdm{{_\ca\bu}_{\rm fd}^{\le 0}}


\nc\isom{\cong} 

\nc{\pone}{{\Bbb C}{\Bbb P}^1}
\nc{\pa}{\partial}
\def\H{\mathcal H}
\def\L{\mathcal L}
\nc{\F}{{\mathcal F}}
\nc{\Sym}{{\goth S}}
\nc{\A}{{\mathcal A}}
\nc{\arr}{\rightarrow}
\nc{\larr}{\longrightarrow}

\nc{\ri}{\rangle}
\nc{\lef}{\langle}
\nc{\W}{{\mathcal W}}
\nc{\uqatwoatone}{{U_{q,1}}(\su)}
\nc{\uqtwo}{U_q(\goth{sl}_2)}
\nc{\dij}{\delta_{ij}}
\nc{\divei}{E_{\alpha_i}^{(n)}}
\nc{\divfi}{F_{\alpha_i}^{(n)}}
\nc{\Lzero}{\Lambda_0}
\nc{\Lone}{\Lambda_1}
\nc{\ve}{\varepsilon}
\nc{\bepsilon}{\bar{\epsilon}}
\nc{\bak}{\bar{k}}
\nc{\phioneminusi}{\Phi^{(1-i,i)}}
\nc{\phioneminusistar}{\Phi^{* (1-i,i)}}
\nc{\phii}{\Phi^{(i,1-i)}}
\nc{\Li}{\Lambda_i}
\nc{\Loneminusi}{\Lambda_{1-i}}
\nc{\vtimesz}{v_\ve \otimes z^m}

\nc{\asltwo}{\widehat{\goth{sl}_2}}
\nc\ag{\widehat{\goth{g}}}  
\nc\teb{\tilde E_\boc}
\nc\tebp{\tilde E_{\boc'}}

\newcommand{\LR}{\bar{R}}
\newcommand{\eeq}{\end{equation}}
\newcommand{\ben}{\begin{eqnarray}}
\newcommand{\een}{\end{eqnarray}}

\title[On the positive part of $U_q(\widehat{sl_2})$ of equitable type]{On the second realization for  the positive part \\ of $U_q(\widehat{sl_2})$ of equitable type} 
\author{Pascal Baseilhac}
\address{Institut Denis-Poisson CNRS/UMR 7013 - Universit\'e de Tours - Universit\'e d'Orl\'eans
Parc de Grammont, 37200 Tours, 
FRANCE}
\email{pascal.baseilhac@idpoisson.fr}

\begin{abstract}   The equitable presentation of the quantum algebra $U_q(\widehat{sl_2})$ is considered. This presentation was originally introduced by T. Ito and P. Terwilliger. In this paper, following Terwilliger's recent works the (nonstandard) positive part of $U_q(\widehat{sl_2})$ of equitable type $U_q^{IT,+}$ and its second realization (current algebra) $U_q^{T,+}$ are introduced and studied. A presentation for $U_q^{T,+}$ is given in terms of a K-operator satisfying a Freidel-Maillet type equation and a condition on its quantum determinant. Realizations of the K-operator in terms of Ding-Frenkel L-operators are considered, from which an explicit injective homomorphism from $U_q^{T,+}$ to a subalgebra of Drinfeld's second realization (current algebra) of $U_q(\widehat{sl_2})$ is derived, and the comodule algebra structure of $U_q^{T,+}$ is characterized.  The central extension of $U_q^{T,+}$ and its relation with Drinfeld's second realization of $U_q(\widehat{gl_2})$ is also described using the framework of Freidel-Maillet algebras. 
\end{abstract}

\maketitle

\vskip -0.5cm


\vspace{2mm}

\section{Introduction}
Originally introduced in \cite{Jim,Dr0}, the quantum affine algebra $U_q(\widehat{sl_2})$ admits a presentation in terms of generators $\{E_i,F_i,K_i^{\pm 1} 
|i=0,1\}$ and relations.  In the literature, this presentation is usually referred as the {\it Drinfeld-Jimbo} or {\it Chevalley} type presentation of $U_q(\widehat{sl_2})$, denoted $U_q^{DJ}$. V. Drinfeld gave also another presentation \cite{Dr}, the so-called {\it Drinfeld's second realization} of $U_q(\widehat{sl_2})$ in terms  generators $\{{\tx}_k^{\pm}, \tho_{\ell},\tK^{\pm 1}, C^{\pm 1/2}
|k\in {\mathbb Z},\ell\in {\mathbb Z}\backslash
\{0\} \}$ and relations, denoted $U_q^{Dr}$. For further analysis, both $U_q^{DJ}$ and $U_q^{Dr}$ are recalled in Appendix \ref{ap:A}. 
A third presentation, initiated by Reshetikhin-Semenov-Tian-Shansky in \cite{RS} and denoted $U_q^{RS}$, takes the form of a  {\it Faddeev-Reshetikhin-Takhtajan  (FRT)} type presentation \cite{FRT89}.  In this case, generating functions for the generators of $U_q^{Dr}$ (and more generally Drinfeld's second realization of $U_q(\widehat{gl_2})$) are the entries of the so-called L-operators, see \cite{DF93} for details. In these definitions, note that the derivation generator is ommited (see \cite[Remark  2, page 393]{CPb}). In the context of mathematics and physics, the presentation $U_q^{DJ}$ and especially $U_q^{Dr}$, $U_q^{RS}$, have played a crucial role in developments of quantum affine algebras, conformal field theory and integrable lattice systems.\vspace{1mm}

In \cite{IT03}, T. Ito and P. Terwilliger obtained a fourth presentation of $U_q(\widehat{sl_2})$ called `equitable', here denoted $U_q^{IT}$, see Theorem \ref{thm2}. It is generated by $\{y_i^\pm,k_i^\pm|i=0,1\}$ subject to the defining relations (\ref{eq:2buq2})-(\ref{eq:2buq7}). An explicit isomorphism $U_q^{IT} \rightarrow U_q^{DJ}$ is known \cite{IT03}, see (\ref{eq:isol1})-(\ref{eq:isol3}). To our knowledge, the relationship between $U_q^{IT}$, $U_q^{Dr}$ and $U_q^{RS}$ has not been investigated.
As a starting point, in this paper we consider the subalgebra of $U_q^{IT}$ 
generated by $\{y_0^+,y_1^+\}$. We denote this subalgebra by $U_q^{IT,+}$
and call it the (nonstandard) positive part of $U_q^{IT}$.
It is known that $U_q^{IT,+}$ has a presentation by generators $\{y_0^+,y_1^+\}$
subject to the q-Serre relations; see (\ref{eq:2buq7}). In a recent work \cite{Ter19}, P. Terwilliger gave another realization - called `alternating' - for an algebra with q-Serre defining relations.
Adapting the results and notations of \cite{Ter19} to $U_q^{IT,+}$, we introduce {\it Terwilliger's second realization of $U_q^{IT,+}$}, denoted $U_q^{T,+}$. It has equitable generators $\{y_{-k}^+,y_{k+1}^+,z^+_{k+1},\tz^+_{k+1}|{k\in {\mathbb N}}\}$ subject to a set of relations displayed in Theorem \ref{thm:Tp}. A PBW basis for $U_q^{T,+}$ is given in Proposition \ref{pbw}. For completeness, following \cite{Ter19b}  the central extension of $U_q^{T,+}$, denoted ${\cal U}_q^{T,+}$, is considered in the last section. See Definitions \ref{ext1}, \ref{ext2}.
\vspace{1mm}

The purpose of this paper is to study the relationship
between $U_q^{T,+}$, its central extension ${\cal U}_q^{T,+}$ and certain subalgebras of  $U_q^{RS}$, $U_q^{Dr}$ and $U_q(\widehat{gl_2})$'s counterparts.
The main result is a Freidel-Maillet type presentation \cite{FM91}  for $U_q^{T,+}$, see Theorem \ref{thm1}. In this presentation, the generators of $U_q^{T,+}$ arise  as coefficients of generating functions characterizing the entries of a K-operator that satisfies a Freidel-Maillet type equation and a quantum determinant equation.
A K-operator that reads as a quadratic combination of L-operators of $U_q^{RS}$ (known in the literature as Sklyanin's dressed operators \cite{Skly88}) is derived, see Lemma \ref{lem:ktmz} and (\ref{Kmu}).  Using this relation between K and L-operators, the following results are obtained in a straightforward manner: generating functions of equitable generators of $U_q^{T,+}$ and  Drinfeld generators of $U_q^{Dr}$ are related through the map $\nu$, see Proposition \ref{map1} and Example \ref{exi1}; $U_q^{T,+}$ is interpreted as a comodule algebra. See Proposition \ref{prop:delta} and Lemma \ref{coprodform}. Relaxing the condition on the quantum determinant, the central extension ${\cal U}_q^{T,+}$ is studied along the same lines using a Freidel-Maillet type presentation. 
 See Theorem \ref{thm3}, Propositions \ref{propqdet}, \ref{prop:YZmap} and Corollary \ref{cor:Cmap}. \vspace{1mm} 

The text is organized as follows. In Section \ref{sec:eq}, the equitable presentation $U_q^{IT}$, its nonstandard positive part $U_q^{IT,+}$ and Terwilliger's second realization $U_q^{T,+}$ are introduced. Then, following recent results \cite{Bas20} a Freidel-Maillet type presentation for $U_q^{T,+}$ is proposed. In Section \ref{sec:FM}, the analysis of \cite[Subsection 5.2]{Bas20} is extended:
K-operator solutions of a Freidel-Maillet type equation are constructed, from which an injective homomorphism $\nu: U_q^{T,+} \rightarrow {U'_q}^{Dr,\triangleright,+}$ is derived, where ${U'_q}^{Dr,\triangleright,+}$ is a subalgebra for $U_q^{Dr}$.
Using the Freidel-Maillet type presentation, it is also shown 
 that $U_q^{T,+}$ admits a (left) comodule algebra structure $\delta: U_q^{T,+} \rightarrow {U'_q}^{Dr,\triangleright,+} \otimes U_q^{T,+}$. For the specialization $C=1$, the image of the equitable generators by the corresponding (left) coaction map is given. In Section \ref{secextU}, for completeness a Freidel-Maillet type presentation for ${\cal U}_q^{T,+}$ is given. An injective homomorphism  $\mu: {\cal U}_q^{T,+} \rightarrow {U'_q}(\widehat{gl_2})^{\triangleright,+}$ is derived, where ${U'_q}(\widehat{gl_2})^{\triangleright,+}$ is a subalgebra of $U_q(\widehat{gl_2})$. In particular, in terms of the Drinfeld's generators of $U_q(\widehat{gl_2})$ the image of the quantum determinant by $\mu$ enjoys a simple factorized structure, see (\ref{Cmap}) or (\ref{Cmapbis}). In the last section, the results here presented together with \cite{Bas20} are summarized and some perspectives are given.
\vspace{1mm}

{\bf Nota bene.} In a recent paper \cite{Ter21b}, the relationship between the positive part of $U_q(\widehat{sl_2})$ denoted $U_q^+$ (see comments around eq. (\ref{defAB})) and its central extension ${\cal U}_q^+$ is studied in details using the framework of generating functions. Explicit relations between generating functions in terms of Damiani's root vectors for $U_q^+$ and generating functions for the alternating generators of ${\cal U}_q^+$ are obtained. For the choice $\bar\epsilon_\pm=0$, fixing $\bar k_\pm$ and $\lambda$ according to the normalizations chosen in \cite{Ter21b} and using Beck's correspondence between Drinfeld's generators and root vectors \cite{Beck} (see (\ref{imr1})-(\ref{imr2})), it can be readily checked that the expressions given in Proposition  \ref{prop:YZmap} with (\ref{im1})-(\ref{im4}), and eq. (\ref{Cmapbis}) match with the expressions given in \cite[Propositions 9.1, 9.3]{Ter21b} and \cite[eq. (65)]{Ter21b}, respectively. \vspace{1mm}

{\bf Notation.}
{\it
Recall the natural numbers ${\mathbb N} = \{0, 1, 2, \cdots\}$ and integers ${\mathbb Z} = \{0, \pm 1, \pm 2, \cdots\}$.
${\mathbb C}(q)$ denotes the field of rational functions in an indeterminate $q$. The $q$-commutator $\big[X,Y\big]_q=qXY-q^{-1}YX$ is introduced. We denote $[x]= (q^x-q^{-x})/(q-q^{-1})$.  } \vspace{1mm}

\section{The equitable subalgebra $U_q^{IT,+}$ and second realization $U_q^{T,+}$}\label{sec:eq}
In this section, the equitable presentation of $U_q(\widehat{sl_2})$ introduced in \cite{IT03} is recalled, and an isomorphism $U_q^{IT} \rightarrow U_q^{DJ}$ is displayed. Then, the positive part $U_q^{IT,+}$ is considered. For its second realization $U_q^{T,+}$ recently introduced in \cite{Ter18}, some properties are recalled. Following \cite{Bas20} a Freidel-Maillet type presentation is given for $U_q^{T,+}$.    
\begin{thm}
\label{thm2} \cite{IT03} 
The quantum affine algebra $U_q(\widehat{sl_2})$ is isomorphic to
the unital associative $\mathbb C(q)$-algebra with {\it equitable} generators $\{y^{\pm}_i,k_i^{{\pm}1}|i=0,1\}$
and the following relations:
\begin{eqnarray}
k_ik^{-1}_i = 
k^{-1}_ik_i &=&  1,\qquad k_0k_1 \ \ \mbox{central},
\label{eq:2buq2}
\\
\frac{\big[y^+_i,k_i\big]_q}{q-q^{-1}} &=& 1,
\qquad
\frac{\big[k_i,y^-_i\big]_q }{q-q^{-1}} = 1,
\qquad
\frac{\big[y^-_i,y^+_i\big]_q}{q-q^{-1}} = 1,
\qquad
\frac{\big[y^+_i,y^-_j\big]_q}{q-q^{-1}} = k^{-1}_0k^{-1}_1,
\qquad i\not=j,
\label{eq:2buq6}
\end{eqnarray}
\begin{eqnarray}
(y^{\pm}_i)^3y^{\pm}_j -  
\lbrack 3 \rbrack_q (y^{\pm}_i)^2y^{\pm}_j y^{\pm}_i 
+\lbrack 3 \rbrack_q y^{\pm}_iy^{\pm}_j (y^{\pm}_i)^2 - 
y^{\pm}_j (y^{\pm}_i)^3 =0, \qquad i\not=j.
\label{eq:2buq7}
\end{eqnarray}
\end{thm}
We call $U_q^{IT}$ the {\it Ito-Terwilliger} or {\it equitable} presentation of $U_q(\widehat{sl_2})$.\vspace{1mm}

An isomorphism $U_q^{IT}\rightarrow U_q^{DJ}$ is given in \cite{IT03}, where $U_q^{DJ}$ is the Drinfeld-Jimbo presentation of  $U_q(\widehat{sl_2})$ recalled in Appendix  \ref{ap:A}. Namely, 
\begin{eqnarray}
k^{\pm 1}_i &\mapsto & K^{\pm 1}_i\ ,\label{eq:isol1}\\
y^-_i &\mapsto &  K^{-1}_i+(q-q^{-1})F_i\ ,\label{eq:isol2} \\
y^+_i &\mapsto &  K^{-1}_i-q(q-q^{-1})K^{-1}_i E_i\ .\label{eq:isol3}
\end{eqnarray}
In this paper, we focus on the following subalgebra.  
\begin{defn} 
 $U_q^{IT,+}$ is the subalgebra of $U_q^{IT}$ generated by $\{y^+_0,y^+_1\}$.  We call $U_q^{IT,+}$ the positive part of $U_q(\widehat{sl_2})$ of equitable type.
\end{defn}
By (\ref{eq:2buq7}), this subalgebra has a presentation by generators
$\{y^+_0,y^+_1\}$ subject to the q-Serre relations.
%
%
Let $U_q^{DJ,+}$ (resp. $U_q^{DJ,-}$) denote the subalgebra of $U_q^{DJ}$ generated by $E_0,E_1$ (resp. $F_0,F_1$); See Appendix \ref{ap:A}. In the literature, $U_q^{DJ,+}$ (resp. $U_q^{DJ,-}$) is usually called the positive (resp. negative) part of $U_q(\widehat{sl_2})$. For this reason, the definition above of positive part of $U_q(\widehat{sl_2})$ is nonstandard. The negative part of $U_q(\widehat{sl_2})$ of equitable type - denoted by $U_q^{IT,-}$ - can be introduced similarly. It is generated by $\{y_0^-,y_1^-\}$. Another subalgebra is the  `Cartan part' denoted $U_q^{IT,0}$, generated by $\{k^{\pm 1}_0,k^{\pm 1}_1\}$. 
%
\begin{defn} ${U'_q}^{DJ,+}$ (resp. ${U'_q}^{DJ,-}$) denotes the subalgebra of $U_q^{DJ}$  generated by $U_q^{DJ,+}$  (resp. $U_q^{DJ,-}$) and $\{K^{\pm 1}_0,K^{\pm 1}_1\}$. 
\end{defn}
By (\ref{defUqDJp}), (\ref{eq:2buq7}), it follows:
\begin{rem}\label{prop:UqITDJ} An injective homomorphism $U_q^{IT,+}\rightarrow {U'}_q^{DJ,+}$ is given by (\ref{eq:isol3}).
\end{rem}

From the point of view of generators and relations, $U_q^{DJ,+}$ and $U_q^{IT,+}$ are exactly the same, up to isomorphism. However, it is seen that their embeddings into  ${U_q}^{DJ}$ essentially differ. To avoid any confusion in further discussions, let us introduce the algebra $U_q^+$ with fundamental generators $A,B$ and q-Serre defining relations:
\beqa
&& \big[A, \big[A, \big[A,B \big]_{q} \big]_{q^{-1}} \big]=0\ ,\quad
 \big[B, \big[B, \big[B,A \big]_{q} \big]_{q^{-1}} \big]= 0\ . \label{defAB}
\eeqa
According to previous definitions,
\begin{lem}\label{lem:1} There exists an algebra isomorphism $U_q^+ \rightarrow  U_q^{DJ,+}$ that sends $A\mapsto E_0$ and $B\mapsto E_1$.
\end{lem}
\begin{lem}\label{lem:2} There exists an algebra isomorphism $U_q^+ \rightarrow  U_q^{IT,+}$ that sends $A\mapsto y^+_0$ and $B\mapsto y_1^+$.
\end{lem}

For $U_q^+$, Terwilliger recently gave a new presentation called {\it alternating}. The alternating presentation consists of infinitly many countable alternating elements called the alternating generators satisfying certain relations \cite{Ter18,Ter19}.  \vspace{1mm} 

For $A\mapsto E_0$ and $B\mapsto E_1$, the alternating presentation produces a new `current' realization for $U_q^{DJ,+}$ besides the known one in terms of Drinfeld generators and relations \cite{Beck}. In this case, an explicit isomorphism between Terwilliger's alternating algebra and certain alternating subalgebras of $U_q^{Dr}$ is established in \cite{Bas20}. For the precise relation between the alternating and Drinfeld's generators, see \cite[Subsection 5.2.3]{Bas20}. Using the correspondence (\ref{isoCP}), in particular one finds $A \mapsto \tx_1^-\tK^{-1}$ and $B\mapsto\tx_0^+$.\vspace{1mm}

For $A\mapsto y^+_0$ and $B\mapsto y_1^+$, the alternating presentation produces similarly a new realization
for $U_q^{IT,+}$. 
\begin{thm}
\label{thm:Tp} (see \cite{Ter19}) 
$U_q^{IT,+}$  is isomorphic to the unital associative $\mathbb C(q)$-algebra with {\it equitable} generators $\{y_{-k}^+,y_{k+1}^+,z_{k+1}^+,\tz_{k+1}^+|k\in {\mathbb N}\}$ subject to the following relations 
\begin{align}
&
 \lbrack  {y}^+_1,  {y}^+_{-k}\rbrack= 
\lbrack  {y}^+_{k+1},  {y}^+_0\rbrack=
\frac{({z}^+_{k+1} -  \tilde{z}^+_{k+1})}{q+q^{-1}},
\label{def1}
\\
&
\lbrack  {y}^+_1,  \tilde{z}^+_{k+1}\rbrack_q= 
\lbrack {z}^+_{k+1},  {y}^+_1\rbrack_q= 
 \bar\rho{y}^+_{k+2},
\label{def2}
\\
&
\lbrack \tilde{z}^+_{k+1},  {y}^+_0\rbrack_q= 
\lbrack  {y}^+_0, {z}^+_{k+1}\rbrack_q= 
\bar\rho {y}^+_{-k-1},
\label{def3}
\\
&
\lbrack  {y}^+_{k+1}, {y}^+_{\ell+1}\rbrack=0,  \qquad 
\lbrack  {y}^+_{-k}, {y}^+_{-\ell}\rbrack= 0,
\label{def4}
\\
&
\lbrack  {y}^+_{k+1}, {y}^+_{-\ell}\rbrack+
\lbrack {y}^+_{-k}, {y}^+_{\ell+1}\rbrack= 0,
\label{def5}
\\
&
\lbrack  {y}^+_{k+1},  \tilde{z}^+_{\ell+1}\rbrack+
\lbrack \tilde{z}^+_{k+1}, {y}^+_{\ell+1}\rbrack= 0,
\label{def6}
\\
&
\lbrack {y}^+_{k+1},  {z}^+_{\ell+1}\rbrack+
\lbrack  {z}^+_{k+1}, {y}^+_{\ell+1}\rbrack= 0,
\label{def7}
\\
&
\lbrack  {y}^+_{-k},  \tilde{z}^+_{\ell+1}\rbrack+
\lbrack   \tilde{z}^+_{k+1},{y}^+_{-\ell}\rbrack= 0,
\label{def8}
\\
&
\lbrack  {y}^+_{-k},  {z}^+_{\ell+1}\rbrack+
\lbrack  {z}^+_{k+1}, {y}^+_{-\ell}\rbrack= 0,
\label{def9}
\\
&
\lbrack  \tilde{z}^+_{k+1},  \tilde{z}^+_{\ell+1}\rbrack=0,
\qquad 
\lbrack {z}^+_{k+1},  {z}^+_{\ell+1}\rbrack= 0,
\label{def10}
\\
&
\lbrack {z}^+_{k+1},  \tilde{z}^+_{\ell+1}\rbrack+
\lbrack  \tilde{z}^+_{k+1},  {z}^+_{\ell+1}\rbrack= 0\ ,
\label{def11}
\end{align}
and the condition ($z^+_0=\tilde{z}^+_0=\bar\rho/(q-q^{-1})$):
\beqa
\bar\rho(q+q^{-1}) \sum_{k=0}^{n} q^{-n+2k} y_{k+1}^+y_{-n+k}^+ - \sum_{k=0}^{n+1}q^{2k-n-1} z_k^+\tilde{z}_{n+1-k}^+ =0\ , \quad n\geq 0\ ,\label{condeq}
\eeqa
with 
\beqa
\bar\rho = q^{-1}(q^2-q^{-2})^2 \ .\label{def:rho}
\eeqa

\end{thm}
This algebra is denoted $U_q^{T,+}$. We call $U_q^{T,+}$  {\it Terwilliger's second realization} of $U_q^{IT,+}$.
For a proof of the above Theorem, we refer the reader to \cite{Ter19,Ter19b} for all details. Compared with the conventions in \cite{Ter19,Ter19b}, the following substitutions are considered:
\beqa
&&y^+_{k+1} \rightarrow W_{-k}\ , \quad y^+_{-k} \rightarrow W_{k+1}\ , \nonumber\\
&&z^+_{k+1} \rightarrow q^{-1}(q^2-q^{-2})G_{k+1}\ , \quad  \tilde{z}^+_{k+1} \rightarrow q^{-1}(q^2-q^{-2}){\tilde G}_{k+1}\ ,\nonumber\\
&&\bar\rho \rightarrow q^{-1}(q^2-q^{-2})(q-q^{-1})\ .\nonumber
\eeqa
\begin{rem}
The relations (\ref{def1})-(\ref{def11}) coincide with the defining relations for the alternating central extension of $U_q^+$, denoted ${\cal U}_q^+$, see \cite[Definition 3.1]{Ter19b}. To get $U_q^{T,+}$ from ${\cal U}_q^+$, the additional relation (\ref{condeq}) is asserted, see \cite[Lemma 2.8]{Ter19b}.  
 \end{rem}

Note that there exists an automorphism $\sigma$ and an antiautomorphism $S$ (see \cite[Proposition 5.3]{Ter19})  such that:
\beqa
\sigma: && {y}^+_{-k}\mapsto {y}^+_{k+1}\ ,\quad {y}^+_{k+1}\mapsto {y}^+_{-k}\ ,\quad {z}^+_{k+1}\mapsto \tilde{z}^+_{k+1}\ ,\quad \tilde{z}^+_{k+1}\mapsto {z}^+_{k+1}\ ,\label{sig}\\
S: && {y}^+_{-k}\mapsto {y}^+_{-k}\ ,\quad {y}^+_{k+1}\mapsto {y}^+_{k+1}\ ,\quad {z}^+_{k+1}\mapsto \tilde{z}^+_{k+1}\ ,\quad \tilde{z}^+_{k+1}\mapsto {z}^+_{k+1}\ .\label{autS}
\eeqa
\vspace{1mm}

The following proposition is a straightforward adaptation of \cite[Theorem 10.2]{Ter19b}.
\begin{prop}\label{pbw}   A PBW basis for $U_q^{T,+}$ is obtained by its equitable generators
\beqa
\{y^+_{-k}\}_{k\in {\mathbb N}}\ ,\quad \{z^+_{n+1}\}_{n\in {\mathbb N}}\ ,\quad \{y^+_{\ell+1}\}_{\ell\in {\mathbb N}}  \nonumber
\eeqa
in any linear order $<$  that satisfies 
\beqa
&&y^+_{-k}<   z^+_{n+1} < y^+_{\ell+1}\ ,\qquad k,\ell,n \in {\mathbb N}\ .\nonumber
\eeqa
\end{prop}
Combining $\sigma$, $S$, 
  other examples of PBW bases can be obtained. 
\vspace{1mm}

The equitable generators of $U_q^{T,+}$ are polynomials in $y_0^+,y_1^+$. Explicit expressions are obtained recursively adapting \cite[Lemma 2.9]{Ter19}. For instance, besides $y_0^+,y_1^+$, the first generators read:
\beqa
\quad && \ z^+_{1} = q{y^+_1}y^+_0-q^{-1}{y^+_0}y^+_1 \ ,\label{zp1}\\
&&y^+_{-1} = \frac{1}{\bar\rho}\left( (q^2+q^{-2})y_0^+y_1^+y_0^+ -(y_0^+)^2y^+_1 - y^+_1 (y_0^+)^2\right) \ \label{ypm1}
\eeqa
and $y^+_2=\sigma(y^+_{-1})$, ${\tilde z}^+_1=\sigma(z^+_{1})$.  Thus, the algebra $U_q^{T,+}$ has a natural $\mathbb N^2$-grading. Define ${\rm deg}: U_q^{T,+} \rightarrow \mathbb N \times \mathbb N$.  For instance, ${\rm deg}(y^+_{0})=(1,0)$ and ${\rm deg}(y^+_{1})=(0,1)$. More generally, ${\rm deg}(y^+_{-k}) = (k+1,k)$, ${\rm deg}(y^+_{k+1}) = (k,k+1)$, ${\rm deg}(z^+_{k+1})={\rm deg}(\tilde{z}^+_{k+1}) = (k+1,k+1)$.\vspace{1mm}

The algebra $U_q^{T,+}$ admits a presentation in the form of a quadratic algebra of Freidel-Maillet type \cite{FM91}, which can be viewed as a limiting case of a  reflection algebra introduced in the context of boundary quantum inverse scattering theory \cite{Cher,Skly88}. Let  $R(u)$ be the quantum $R-$matrix
defined by \cite{Baxter}
\begin{align}
R(u) =\left(
\begin{array}{cccc} 
 uq -  u^{-1}q^{-1}    & 0 & 0 & 0 \\
0  &  u -  u^{-1} & q-q^{-1} & 0 \\
0  &  q-q^{-1} & u -  u^{-1} &  0 \\
0 & 0 & 0 & uq -  u^{-1}q^{-1}
\end{array} \right) \ ,\label{R}
\end{align}
where $u$  is an indeterminate,  called  `spectral parameter' in the literature on integrable systems, and 
deformation parameter $q$. It is known that $R(u)$ satisfies the quantum Yang-Baxter equation in the space ${\cal V}_1\otimes {\cal V}_2\otimes {\cal V}_3$, with ${\cal V}\equiv {\mathbb C}^2$. Using the standard notation 
\beqa
R_{ij}(u)\in \mathrm{End}({\cal V}_i\otimes {\cal V}_j),\label{notR}
\eeqa
the Yang-Baxter equation reads 
\begin{align}
R_{12}(u/v)R_{13}(u)R_{23}(v)=R_{23}(v)R_{13}(u)R_{12}(u/v)\ .\label{YB}
\end{align}
 As usual, intoduce the permutation operator  $P=R(1)/(q-q^{-1})$. Here, note that  $R_{12}(u)=PR_{12}(u)P=R_{21}(u)$.
In addition to (\ref{R}), define:
\beqa
 R^{(0)}=diag(1,q^{-1},q^{-1},1)\ .\label{R0}
\eeqa
Define the generating functions:
\begin{align}
{\cY}_+(u)=\sum_{k\in {\mathbb N}}{y}^+_{k+1}U^{-k-1} \ , \quad {\cY}_-(u)=\sum_{k\in  {\mathbb N}}{y}^+_{-k}U^{-k-1} \ ,\label{c1}\\
 \quad {\cZ}_+(u)=\sum_{k\in {\mathbb N}}{\tilde z}^+_{k+1}U^{-k-1} \ , \quad {\cZ}_-(u)=\sum_{k\in {\mathbb N}}{z}^+_{k+1}U^{-k-1} \ ,\label{c2}
\end{align}
where  the shorthand notation $U=qu^2/(q+q^{-1})$
is used.  Let $\bar k_\pm \in {\mathbb C}(q)$ such that
\beqa
\bar\rho={\bar k}_+{\bar k}_-(q+q^{-1})^2\ .\label{rho}
\eeqa

For the alternating central extension of $U_q^+$, a Freidel-Maillet type presentation has been proposed in \cite[Theorem 3.1]{Bas20}. It is given in terms of a K-operator satisfying a Freidel-Maillet type equation. To get $U_q^+$, a condition on the quantum determinant of the K-operator is required. 
\begin{thm}\label{thm1} The algebra $U_q^{T,+}$ has a presentation of Freidel-Maillet type.
 Let $K(u)$ be a square matrix such that 
\beqa
&&  K(u)=
       \begin{pmatrix}
      uq \cY_+(u) &  \frac{1}{{\bar k}_-(q+q^{-1})}\cZ_+(u) + \frac{{\bar k}_+(q+q^{-1})}{(q-q^{-1})} \\
     \frac{1}{{\bar k}_+(q+q^{-1})}\cZ_-(u) + \frac{{\bar k}_-(q+q^{-1})}{(q-q^{-1})}  & uq \cY_-(u) 
      \end{pmatrix} \label{K}
\eeqa
with (\ref{c1})-(\ref{c2}).
The defining relations are given by:
\begin{align} R(u/v)\ (K(u)\otimes {\mathbb I})\ R^{(0)}\ ({\mathbb I} \otimes K(v))\
= \ ({\mathbb I} \otimes K(v))\  R^{(0)}\ (K(u)\otimes {\mathbb I})\ R(u/v)\ 
\label{RE} \end{align}
and\footnote{As usual, `$\rm tr_{12}$' stands for the trace over $\cal V_1 \otimes \cal V_2$. Also, we denote $P^{-}_{12}=(1-P)/2$.}
\beqa
\tr_{12}\big(P^{-}_{12}(K(u)\otimes {\mathbb I})\ R^{(0)} ({\mathbb I} \otimes K(uq))\big)= -\frac{\bar\rho}{(q-q^{-1})^2}\ .\label{condqdet}
\eeqa
\end{thm}
\begin{proof} By specializing some results of \cite{Bas20}, the proof follows. The first part of the proof concerns the equivalence between (\ref{def1})-(\ref{def11}) and (\ref{RE}). Recall the defining relations of the alternating central extension of $U_q^+$ (i.e. ${\cal U}_q^+$) given in \cite[Definition 2.1]{Bas20}. Observe that they coincide with the subset of relations (\ref{def1})-(\ref{def11}) upon the substitution:
\beqa
&&\tW_{-k}\rightarrow y^+_{k+1}\ ,\qquad \tW_{k+1}\rightarrow y^+_{-k}\ ,\label{cor1}\\
&&\tG_{k+1}\rightarrow {\tilde z}^+_{k+1}, \qquad \tilde{\tG}_{k+1}\rightarrow z^+_{k+1}\ .\label{cor2}
\eeqa 
Now, by \cite[Theorem 3.1]{Bas20} it is known that ${\cal U}_q^+$ admits a Freidel-Maillet type presentation given by a K-operator satisfying (\ref{RE}). So, for the K-operator (\ref{K}), the relations (\ref{def1})-(\ref{def11}) are equivalent to (\ref{RE}).

The second part of the proof concerns the equivalence between (\ref{condeq}) and (\ref{condqdet}). By \cite[Proposition 3.3]{Bas20}, the l.h.s of (\ref{condqdet}) is the so-called quantum determinant that generates the center of ${\cal U}_q^+$. For convenience, define
\beqa
{\cal C}(u) = (q-q^{-1})u^2q^2\cY_+(u)\cY_-(uq)  - \frac{(q-q^{-1})}{\bar\rho} \cZ_-(u)\cZ_+(uq)  - \cZ_-(u) - \cZ_+(uq)\ .\label{Zvee}
\eeqa
Inserting (\ref{K}) into the l.h.s. of (\ref{condqdet}), the quantum determinant reduces to:
\beqa
\tr_{12}\big(P^{-}_{12}(K(u)\otimes {\mathbb I})\ R^{(0)} ({\mathbb I} \otimes K(uq))\big)= \frac{1}{2(q-q^{-1})}\left( {\cal C}(u) + \sigma({\cal C}(u)) -\frac{2\bar\rho}{(q-q^{-1})}\right)\ .\label{condqdetred}
\eeqa
Using the exchange relations between the generating functions (\ref{c1})-(\ref{c2}) extracted from (\ref{RE}), one shows $\sigma({\cal C}(u)) = {\cal C}(u)$.
Thus, the condition (\ref{condqdet}) is equivalent to:
\beqa
 {\cal C}(u) =0 \ .\label{Zveecond}
\eeqa
Extracting the set of constraints on the coefficients of the generating function ${\cal C}(u)$, one gets (\ref{condeq}). 
\end{proof}

Note that  eqs. (\ref{RE}), (\ref{condqdet}), are left invariant under the transformation $(u,v)\mapsto (\lambda u,\lambda v)$ for $\lambda$ invertible and $[\lambda,U_q^{T,+}]=0$. This property will be used in further analysis.

\section{Relating  Terwilliger's and Drinfeld's second realizations}\label{sec:FM}
It is natural to ask for the precise relationship between the equitable and Drinfeld's generators. As shown in this section, the Freidel-Maillet type presentation of Theorem \ref{thm1} combined with the framework of FRT algebras \cite{FRT89,RS,DF93} gives a suitable framework for answering this question. In addition, it provides a tool for constructing left or right coaction maps that ensure a comodule algebra structure for $U_q^{T,+}$.\vspace{1mm}

Below, as a preliminary the FRT presentation for $U_q^{Dr}$ is first recalled, and  Drinfeld type `alternating' subalgebras $\{U_q^{Dr,a,\pm}\}$, their extensions $\{{U'_q}^{Dr,a,\pm}\}$ for $a= \triangleright,\triangleleft$, are introduced. Then, a K-operator satisfying a Freidel-Maillet type equation is constructed, and used to derive an injective homomorphism $\nu: \ U_q^{T,+} \rightarrow {U'_q}^{Dr,\triangleright,+}$. Using the comodule algebra structure of the Freidel-Maillet type presentation, a left coaction map $\delta: \ U_q^{T,+} \rightarrow {U'_q}^{Dr,\triangleright,+} \otimes U_q^{T,+}$ is also derived. For the specialization  $\bar\delta: \ U_q^{T,+} \rightarrow {U'_q}^{Dr,\triangleright,+}/_{C=1} \otimes U_q^{T,+}$ the image of the generating functions for the equitable generators (\ref{c1}), (\ref{c2}) is given.

\subsection{FRT presentation} For the quantum affine Lie algebra  $U_q(\widehat{gl_2})$, a FRT presentation is known \cite{RS,DF93}. Define the $R$-matrix:
\begin{equation}\label{def:r}
\tilde R(z)=\begin{pmatrix} 1
      &0&0&0\\
       0&\frac{z-1}{zq-q^{-1}}& \frac{ z(q-q^{-1})}{zq-q^{-1}}&0\\
       0&  \frac{(q-q^{-1})}{zq-q^{-1}} & \frac{ z-1}{zq-q^{-1}} &0\\
        0&0&0&1
      \end{pmatrix} \ 
\end{equation}
where $z$  is an indeterminate. It is known that $\tilde R(z)$ satisfies the quantum Yang-Baxter equation 
\begin{align}
\tilde R_{12}(z_1/z_2)\tilde R_{13}(z_1)\tilde R_{23}(z_2)=\tilde R_{23}(z_2)\tilde R_{13}(z_1) \tilde R_{12}(z_1/z_2)\ .\label{YB}
\end{align}
In terms of ${\tilde R}(z)$, the permutation operator reads $P=\tilde R(1)$. Note that  $\tilde R_{12}(z)=\tilde R_{21}^{t_1t_2}(z)$. 
\begin{thm}\label{def:UqRS}\cite{RS,DF93} $U_q(\widehat{gl_2})$ admits a FRT presentation given by a unital associative algebra with generators  $\{{\tx}_k^{\pm}, \tk^+_{j,-\ell}, \tk^-_{j,\ell}, q^{\pm c/2} |k\in {\mathbb Z},\ell\in {\mathbb N},j=1,2 \}$. The  generators  $q^{\pm c/2}$ are central and mutally inverse. Define:
\beqa
&&  L^\pm(z)=
       \begin{pmatrix}
     \tk_1^\pm(z)  &    \tk_1^\pm(z)  \tf^\pm(z)    \\
   \te^\pm(z) \tk_1^\pm(z)     &   \tk_2^\pm(z) +   \te^\pm(z)  \tk_1^\pm(z)    \tf^\pm(z)   
      \end{pmatrix} \ \label{Lpm}
\eeqa
in terms of  the generating functions in the indeterminate $z$:
\beqa
{\te}^+(z)&=& (q-q^{-1})\sum_{k=0}^\infty q^{k(c/2-1)} {\tx}_{-k}^- z^{k} \ , \quad {\te}^-(z)=-(q-q^{-1})\sum_{k=1}^\infty q^{k(c/2+1)}{\tx}^-_{k}z^{-k} \ ,\label{eq:cuDre}\\
{\tf}^+(z)&=&(q-q^{-1})\sum_{k=1}^\infty q^{-k(c/2+1)}{\tx}^+_{-k}z^{k} \ , \quad {\tf}^-(z)=-(q-q^{-1})\sum_{k=0}^\infty q^{-k(c/2-1)}{\tx}^+_{k}z^{-k} \ ,\label{eq:cuDrf}\\
\quad  {\tk}_{j}^+(z)&=&\sum_{k=0}^\infty{{\tk}}^+_{j,-k}z^{k}\  \ ,\quad \qquad \qquad {\tk}_{j}^-(z)=\sum_{k=0}^\infty{{\tk}}^-_{j,k}z^{-k}\ ,\quad j=1,2\ .\label{eq:cuDrk}
\eeqa
The defining relations are the following:
\beqa
\tk^+_{i,0}\tk^-_{i,0}&=& \tk^-_{i,0}\tk^+_{i,0}=1 \ ,\label{YBApm0}\\ 
 \tilde  R(z/w)\ (L^\pm(z)\otimes {\mathbb I})\ ( {\mathbb I} \otimes L^\pm(w))  &=&  ( {\mathbb I} \otimes L^\pm(w))\  (L^\pm(z)\otimes {\mathbb I})   \  \tilde R(z/w)  \ ,  \label{YBApm1}\\
\tilde R(q^{c}z/w)\ (L^+(z)\otimes {\mathbb I})\ ( {\mathbb I} \otimes L^-(w))  &=&  ( {\mathbb I} \otimes L^-(w))\  (L^+(z)\otimes {\mathbb I})   \ \tilde R(q^{-c}z/w)  \ .  \label{YBApm2}
 \eeqa
For (\ref{YBApm1}), the expansion direction of $ \tilde R(z/w)$ can be chosen in $z/w$ or $w/z$, but for 
(\ref{YBApm1}) the expansion direction is only in $z/w$.
The Hopf algebra structure is characterized as follows. The coproduct\footnote{The index $[j]$ characterizes the `quantum space' $V_{[j]}$ on which the entries of $L^\pm(z)$ act. With respect to the ordering $V_{[1]}\otimes V_{[2]}$, one has:
\beqa
((T)_{[\textsf 1]}(T')_{[\textsf 2]})_{ij} =\sum_{k=1}^2  (T)_{ik}\otimes (T')_{kj}  \ .
\eeqa
}  
$\Delta$, antipode ${\cal S}$ and counit ${\cal E}$ are such that:
\beqa
&& \Delta(L^\pm(z)) = (L^\pm(zq^{\pm(1\otimes c/2)}))_{[\textsf 1]} (L^\pm(zq^{\mp(c/2\otimes 1)}))_{[\textsf 2]} \ , \label{coprodUqgl2}\\
&& {\cal S}(L^\pm(z))=L^\pm(z)^{-1}\ ,\quad {\cal E}(L^\pm(z))={\mathbb I}\ . \label{counitUqgl2}
\eeqa
\end{thm}

The complete isomorphism between the FRT presentation of Theorem \ref{def:UqRS} and Drinfeld second presentation of $U_q(\widehat{gl_2})$ is given in \cite[Section 4]{Jing} (see also \cite{FMu}). Following \cite[Section 4]{Jing}, introduce the generating functions
\beqa
\tk_i^\pm(z)=\tk_{i,0}^{\pm } \exp\left( \pm (q-q^{-1})\sum_{n= 1}^\infty a_{i,\mp n} z^{\pm n}\right) \  \label{kpmz}
\eeqa
in terms of the new generators $a_{i,\mp n}$. In terms of Drinfeld generators $\tho_m$, the new generators $a_{1,m},a_{2,m}$ decompose as:
\beqa
a_{1,m}= \frac{1}{q^m+q^{-m}}(\tho_m + \gamma_m)\ ,\label{a1mbis}\qquad
a_{2,m}=  -\frac{1}{q^m+q^{-m}}(q^{2m}\tho_m -  \gamma_m)\ \label{a2mbis}\ ,
\eeqa
where $\gamma_m$ are central elements of $U_q(\widehat{gl_2})$. For our purpose, introduce the surjective map $\gamma'_D:  U_q(\widehat{gl_2})  \rightarrow  U_q^{Dr}$ that is defined as follows. Let $\gamma'_m$ be Laurent polynomials in $C^{1/2}$, that will be specified later on. We define:
\beqa
&&\gamma'_D( q^{c/2})  \mapsto C^{1/2}   \ ,\label{gamD0}\\
&& \gamma'_D(\tx_k^\pm) \mapsto  \tx_k^\pm \ ,\label{gamD1}\\
&&\gamma'_D( a_{1,m})\mapsto \frac{1}{q^m+q^{-m}}(\tho_m + \gamma'_m) \ ,\qquad \gamma_D(a_{2,m}) \mapsto -\frac{1}{q^m+q^{-m}}(q^{2m}\tho_m - \gamma'_m)\ ,\label{gamD2} \\
&&\gamma'_D(\tk_{2,0}^\mp(\tk_{1,0}^\mp)^{-1}) \mapsto  \tK^{\pm 1} \ ,
\qquad \quad \gamma_D(\tk_{1,0}^\pm\tk_{2,0}^\pm)\mapsto 1 \ .\label{gamD3} 
\eeqa
Note that the map $\gamma'_D$ slightly differs from the map chosen in \cite[eq. (5.62)-(5.65)]{Bas20}.

\subsection{Alternating subalgebras of $U_q^{Dr}$}
Certain `alternating' subalgebras of $U_q(\widehat{sl_2})$ have been introduced in \cite{Bas20}, that are now reviewed for further analysis. 
\begin{defn}\label{defaltsl2q}
\beqa
U_q^{Dr,\triangleright,\pm} &=& \{C^{\mp k/2}\tK^{-1}\tx_{k}^\pm, C^{\pm (k+1)/2}\tx^\mp_{k+1}, \tho_{k+1}|k\in {\mathbb N}\} \ ,  \nonumber\\
U_q^{Dr,\triangleleft,\pm}&=& \{C^{\mp k/2}\tx_{-k}^\pm, C^{\pm (k+1)/2} \tx^\mp_{-k-1}\tK,\tho_{-k-1}| k\in {\mathbb N}\} \ . \nonumber
\eeqa
We call $U_q^{Dr,\triangleright,\pm}$ and $U_q^{ Dr, \triangleleft,\pm}$ the
 right and left alternating subalgebras of  $U_q^{Dr}$.  The subalgebra generated by $\{\tK^{\pm 1},C^{\pm 1/2}\}$ is denoted  $U_q^{Dr,\diamond}$.
\end{defn}

The defining relations of the alternating subalgebras are identified using (\ref{gl1})-(\ref{gl4}). Consider for instance $U_q^{Dr,\triangleright,+}$. If we denote $A_k^{+}= C^{-k/2}\tK^{-1}\tx_k^+$, $A_{\ell}^{-}= C^{\ell/2}\tx_{\ell}^-$, $B_{\ell}=\tK^{-1}\psi_\ell$ and $D_{\ell}=\tho_\ell$, $k\geq 0,\ \ell \geq 1$, using the relations in Appendix \ref{ap:A} one gets the defining relations:
\beqa
&& \big[ D_k,D_\ell\big]= 0 \ ,\quad \big[ D_k,B_\ell\big]= 0\ , \label{alt1}\\
&& \big[ D_k, A_\ell^{\pm} \big]  = \pm \frac{\big[ 2k\big]_q}{k}  A_{k+\ell}^{\pm}\ ,\\
&& A^\pm_{k+1} A^\pm_\ell  -q^{\pm 2} A^\pm_{\ell} A^\pm_{k+1} = q^{\pm 2} A^\pm_{k} A^\pm_{\ell+1} - A^\pm_{\ell+1} A^\pm_k \ ,\label{alt3}\\ && \big[ A^+_k, A^-_\ell   \big]_{q^{-1}}  = \frac{q^{-1} B_{k+\ell}}{q-q^{-1}}   \ .\label{alt4}
\eeqa
%
%
\begin{lem}\label{isoUqDrDJp} $U_q^{Dr,\triangleright,+} \cong U_q^{DJ,+}$. 
\end{lem}
\begin{proof} Following \cite{BCP} introduce the root vectors $\{E_{k\delta+\alpha_i}, E_{\ell\delta}|i=0,1,k\geq 0,\ell \geq 1\}\in U_q^{DJ,+}$ and the elements $\tilde{\psi}_\ell$, $\ell\geq 1$, through the functional equation:
. 
\beqa
1+ (q-q^{-1})\sum_{\ell=1}^{\infty} \tilde{\psi}_\ell z^\ell = \exp \left( (q-q^{-1}) \sum_{\ell=1}^{\infty} E_{\ell\delta} z^\ell  \right)\ .
\eeqa
A comparison between the specialization to $U_q(\widehat{sl_2})$ of the relations in \cite[Proposition 1.2]{BCP} and the relations (\ref{alt1})-(\ref{alt4}) gives the isomorphism $U_q^{Dr,\triangleright,+} \rightarrow U_q^{DJ,+}$:
\beqa
A_k^+ \rightarrow E_{k\delta+\alpha_1}\ , \quad A_{\ell}^- \rightarrow -q^{-2} E_{(\ell-1)\delta+\alpha_0}\ ,\quad B_{\ell}= (q-q^{-1})\tilde{\psi}_{\ell}\ ,\quad D_\ell = E_{\ell\delta}\ .
\eeqa
\end{proof}

The defining relations of the other alternating subalgebras can be similarly written and related with $U_q^{DJ,+}$ or $U_q^{DJ,-}$.\vspace{1mm}

For   $U_q(\widehat{sl_2})$,  it is known that given a certain ordering the elements  $\{{\tx}_k^{\pm}, \tho_{\ell},\tK^\pm, C^{\pm 1/2}\}$ 
 generate a PBW basis. See \cite[Proposition 6.1]{Beck} with \cite[Lemma 1.5]{BCP}.  For the alternating subalgebras, PBW bases follow naturally. Let us choose the ordering:
\beqa
C^{1/2}\tx_1^- <  C\tx_2^- < \cdots < \tho_{1} < \tho_{2}   < \cdots <  C^{-1/2} \tK^{-1}\tx_1^+ <  \tK^{-1} \tx_0^+ \  \label{pbw1}
\eeqa
for $U_q^{Dr,\triangleright,+}$,
whereas for the  subalgebra $U_q^{Dr,\triangleleft,-}$ we choose the ordering:
\beqa
\tx_{0}^- < C^{1/2} \tx_{-1}^-  < \cdots < \tho_{-1} < \tho_{-2}   < \cdots < C^{-1}\tx_{-2}^+ \tK < C^{-1/2}  \tx_{-1}^+\tK \ . \label{pbw2}
\eeqa
It follows:
\begin{prop}\label{pbwalt} The vector space   $U_q^{Dr,\triangleright,+}$  (resp. $U_q^{Dr,\triangleleft,-}$)  has a linear basis consisting of the products $x_1x_2\cdots x_n$ $(n\in {\mathbb N})$ with $x_i\in U_q^{Dr,\triangleright,+}$  (resp. $x_i\in U_q^{Dr,\triangleleft,-}$)   such that  $x_1 \leq x_2 \leq \cdots \leq x_n$. 
\end{prop}
Using the automorphism (\ref{thetq}), PBW bases for $U_q^{Dr,\triangleright,-}$ and $U_q^{Dr,\triangleleft,+}$ are similarly obtained.\vspace{1mm}

Extensions of the alternating subalgebras are now introduced, that will be useful in the analysis below. 
\begin{defn}\label{defaltsl2qp}  ${U'_q}^{Dr,\triangleright,\pm}$ (resp. ${U'_q}^{Dr,\triangleleft,\pm}$) denote the subalgebras of $U_q^{Dr}$ generated by $U_q^{Dr,\triangleright,\pm}$ (resp. $U_q^{Dr,\triangleleft,\pm}$) and $\{\tK^{\pm 1},C^{\pm 1/2}\}$.
\end{defn}
If one considers for instance ${U'_q}^{Dr,\triangleright,+}$, in addition to the relations (\ref{alt1})-(\ref{alt4}) one has:
\beqa
&& \big[ \tho_k,\tK^{\pm 1}\big]= 0 \ , \quad \big[ B_k,\tK^{\pm 1}\big]= 0 \ ,\quad C^{1/2} \ \ \mbox{central}\ ,\label{alt5}\\
&&  \tK A_k^{\pm} \tK^{-1}  =  q^{\pm 2} A_k^{\pm} \ .\label{alt6}
\eeqa

\subsection{The homomorphism $\nu: U_q^{T,+} \rightarrow {U'_q}^{Dr,\triangleright,+}$}
Consider the following Freidel-Maillet type equation (for a non-symmetric R-matrix)
\begin{align} \tilde R_{12}(z/w)\ (\tilde K(z)\otimes {\mathbb I})\ R^{(0)}\ ({\mathbb I} \otimes \tilde K(w))\
= \ ({\mathbb I} \otimes \tilde K(w))\  R^{(0)}\ (\tilde K(z)\otimes {\mathbb I})\ \tilde R_{21}(z/w)\  .\label{RKz}
\end{align}
Assume there exists  a matrix $\tilde K^0(z)$ with scalar entries and two quantum Lax operators $ L(z) ,  L^{0}$, such that the following relations hold (recall that $\tilde R_{21}(z)= P\tilde R_{12}(z)P$):
\beqa \tilde R_{12}(z/w) \  \tilde K^0_1(z) \ R^{(0)}\ \tilde K^0_2(w)\
&=& \  \tilde K^0_2(w)  \ R^{(0)} \ \tilde K^0_1(z) \ \tilde R_{21}(z/w)\  ,\label{RKzinit} \\
\tilde R_{12}(z/w)    L_1(z)  L_2(w)   &=& L_2(w)   L_1(z) \tilde R_{12}(z/w) \ ,\label{RtLpLp}\\
\tilde R_{21}(z/w)   (L^{0})_1(L^{0})_2 &=& (L^{0})_2(L^{0})_1  \tilde R_{21}(z/w) \ ,\label{RL0L0}\\
(L^{0})_1R^{(0)}  L_2(w)  &=&  L_2(w)   R^{(0)} (L^{0})_1\ ,\label{L0R0L}\\
L_1(z) R^{(0)}(L^{0})_2  &=& (L^{0})_2R^{(0)} L_1(z)  \  .\label{LR0L0}
\eeqa
Adapting \cite[Proposition 2]{Skly88},   using the above relations one finds that  :
\beqa
\tilde K(z) \mapsto L(z\lambda) \tilde K^0(z) L^{0}\label{Ktz}
\eeqa
satisfies (\ref{RKz}) provided $\lambda$ is invertible and $[\lambda,U_q(\widehat{gl_2})]=0$. For instance, define:
\beqa
\tilde K^{0}(z) =  \begin{pmatrix} \bar\epsilon_+ & \frac{{\bar k}_+ (q+q^{-1})}{(q-q^{-1})}  \\
    \frac{{\bar k}_- (q+q^{-1})}{(q-q^{-1})}  & \frac{\bar\epsilon_-}{z}
      \end{pmatrix}\  ,\label{Ktilde0}
\eeqa
where ${\bar k}_\pm\in {\mathbb C}(q)$  and $\big[\bar\epsilon_\pm,U_q(\widehat{gl_2})\big]=0$. It satisfies (\ref{RKzinit}). It follows:
\begin{lem}\label{lem:ktmz} The K-operator 
\beqa
\tilde K(z) \mapsto \tilde K^-(z) = L^-(z \lambda) \tilde K^{0}(z)  L^{-,0}\label{Ktmz} \ 
\eeqa
satisfies (\ref{RKz}) for any invertible $\lambda$ such that $[\lambda,U_q(\widehat{gl_2})]=0$.
\end{lem}
\begin{proof} By previous comment, it is sufficient to check that (\ref{RtLpLp})-(\ref{LR0L0}) hold. For the choices 
\beqa
L(z) \mapsto  L^-(z)  \quad \mbox{and} \quad L^{0}\mapsto  L^{-,0}= diag( (\tk_{2,0}^-)^{-1}, (\tk^-_{1,0})^{-1})\ ,\label{Lm0i} 
\eeqa
 eq. (\ref{RtLpLp}) holds by definition and it is checked that eqs. (\ref{RL0L0})-(\ref{LR0L0}) hold.
\end{proof}

The R-matrices $R(u)$ (symmetric) and $\tilde R(z)$ (non-symmetric) given by (\ref{R}) and (\ref{def:r}), respectively, are related through the similarity transformations:
\beqa
\qquad \left(\frac{u}{v}q-\frac{v}{u}q^{-1}\right)^{-1} R_{12}(u/v) &=& \cal M(u)_1  \cal M(v)_2 \tilde R_{12}(u^2/v^2) \cal M(v)_2^{-1} \cal M(u)_1^{-1}\ ,\label{simil}\nonumber\\
&=& {\cal M}(u)_1^{-1}   {\cal M}(v)_2^{-1} \tilde R_{21}(u^2/v^2) {\cal M}(v)_2 {\cal M}(u)_1 \quad \mbox{with} \quad 
 {\cal M}(u)=   \begin{pmatrix}
         u^{-1/2} & 0 \\
    0   &  u^{1/2} 
                   \end{pmatrix} \ .\nonumber
\eeqa
Using this transformation, one relates (\ref{RKz}) to (\ref{RE}): there exists an injective homomorphism from the Freidel-Maillet algebra (\ref{RE}) to the Yang-Baxter algebra (\ref{YBApm0})-(\ref{YBApm2})  given by:
\beqa
K(u) \mapsto  \cal M(u) {\tilde K}^-(qu^2)   {\cal M}(u)\ .  \label{Kmu} 
\eeqa
%
%
%
%
The explicit expression for (\ref{Kmu}) is a generalization of the K-operator in \cite[Lemma 5.15]{Bas20}. Here the difference relies on  the additional elements $\bar\epsilon_\pm\neq 0$ in (\ref{Ktilde0}).  \vspace{1mm}

The map (\ref{Kmu}) allows to establish the precise relation between the equitable generators $ \{{y}^+_{-k}, {y}^+_{k+1},  {z}^+_{k+1}, {\tilde z}^+_{k+1}\}$ and the generators of alternating subalgebras. 
In the expressions below, for normalization convenience we set:
\beqa
&& {\bar k}_+=q^{-1}(q-q^{-1})\ ,\quad {\bar k}_-=q-q^{-1}\ ,\quad \bar\epsilon_+=q+q^{-1}\ ,\quad \bar\epsilon_-=q(q+q^{-1})C^{-1}\ ,\quad \lambda=C^{3/2}\ . \label{parfixed}
\eeqa
\begin{prop} \label{map1}  There exists an injective homomorphism $\nu:\ U_q^{T,+} \rightarrow {U'_q}^{Dr,\triangleright,+}$ such that: 
\beqa
&&{\cY}_+(u)\mapsto g(u) \left(     -{\bar k}_-(q^2+1) (qu^2)^{-1} \sum_{k=0}^\infty q^{k} C^{-k/2}\tK^{-1}{\tx}_{k}^+ (qu^2\lambda)^{-k}   + \bar\epsilon_+(qu^2)^{-1}\tK^{-1} \right)\ ,\label{im1}\\
&&{\cY}_-(u)\mapsto \left( -{\bar k}_+(q^{-2}+1) \sum_{k=0}^\infty q^{k+1} C^{(k+1)/2} {\tx}_{k+1}^- (qu^2\lambda)^{-k-1}  \right. \\
&& \qquad \qquad \quad \left. + \ \bar\epsilon_- q^{-1}(qu^2)^{-1}\left(\psi(u^2\lambda)  + (q-q^{-1})^2 \sum_{k,\ell =0}^\infty q^{k-\ell} C^{(k-\ell+1)/2}  {\tx}_{k+1}^- {\tx}_{\ell}^+ (qu^2\lambda)^{-k-\ell-1}\right) \right)g(u)\ ,\nonumber\\
%
&&{\cZ}_+(u)\mapsto   \left(\frac{\bar\rho}{q-q^{-1}} - \bar\epsilon_- {\bar k}_-q^{-1}(q^2-q^{-2}) (qu^2)^{-1} \sum_{k=0}^\infty q^{-k} C^{-k/2}{\tx}_{k}^+ (qu^2\lambda )^{-k} \right)g(u)  - \frac{\bar\rho}{q-q^{-1}}\ ,\\
&&{\cZ}_-(u)\mapsto g(u) \left(   \frac{\bar\rho}{q-q^{-1}} \tK^{-1}\psi(u^2\lambda)   
- \ \bar\epsilon_+ {\bar k}_+(q^2-q^{-2}) \sum_{k=0}^\infty q^{-k+1} C^{(k+1)/2}\tx_{k+1}^-\tK^{-1} (qu^2\lambda)^{-k-1}\right.\label{im4}\\
&& \qquad \qquad \qquad \qquad \left.  +\  \bar\rho(q-q^{-1})  \sum_{k,\ell =0}^\infty q^{-k+\ell} C^{(k-\ell+1)/2} \tK^{-1}  {\tx}_{k+1}^- {\tx}_{\ell}^+ (qu^2\lambda)^{-k-\ell-1}  \right)  - \frac{\bar\rho}{q-q^{-1}}\ ,\nonumber
\eeqa
where
\beqa
&&g(u) = \exp\left( -(q-q^{-1}) \sum_{n=1}^\infty \frac{(\tho_n + \gamma'_n)}{q^n+q^{-n}} (qu^2\lambda)^{- n}\right)\ \quad \mbox{with}\quad 
\gamma'_n = -\frac{(q-q^{-1})^{2n-1}}{n}\left( \frac{\bar\epsilon_+\bar\epsilon_-\lambda}{\bar\rho q}\right)^n\ .\label{gamp}
\eeqa
\end{prop}
\begin{proof} The first part of the proof concerns the derivation of the expressions on the r.h.s of (\ref{im1})-(\ref{im4}).
Recall Lemma \ref{lem:ktmz}. Then, one expands explicitly (\ref{Ktmz}) using (\ref{Lpm}). Consider for instance the entry $({\tilde K}^-(z))_{11}$, where some commutation relations given in \cite[eqs. (5.55)-(5.56)]{Bas20} are used: 
\beqa
(\tilde K^-(z))_{11} &=&   \frac{{\bar k}_-(q+q^{-1})}{q-q^{-1}} \tk_1^-(z\lambda)\underbrace{\tf^-(z\lambda)  (\tk^-_{2,0})^{-1}}_{= q  (\tk^-_{2,0})^{-1} \tf^-(z\lambda) } + \bar\epsilon_+ \tk_1^-(z\lambda)(\tk^-_{2,0})^{-1}   
\nonumber\\
&=&    \frac{{\bar k}_-(q+q^{-1})}{q-q^{-1}} \!\!\!\!\!\!\!\! \!\!\!\!\!\!\!\!\!\!\!\!\underbrace{\tk_1^-(z\lambda)(\tk^-_{2,0})^{-1}}_{=\tK^{-1}  \exp\left( -(q-q^{-1}) \sum_{n=1}^\infty a_{1,n} (z\lambda)^{- n}\right)  } \!\!\!\!   \!\!\!\!\!\!\!\! \!\!\!\!\!\!\!\! \tf^-(z\lambda) + \bar\epsilon_+ \underbrace{\tk_1^-(z\lambda)(\tk^-_{2,0})^{-1}}_{}       \qquad \mbox{by} \quad (\ref{kpmz})\nonumber\ .
\eeqa
Inserting  (\ref{eq:cuDrf}),  one gets:
\beqa
(\tilde K^-(z))_{11} &=&  -{\bar k}_-(q+q^{-1}) \exp\left( -(q-q^{-1}) \sum_{n=1}^\infty a_{1,n} (z\lambda)^{- n}\right) \sum_{k=0}^\infty q^{k} q^{-ck/2}\tK^{-1}{\tx}_{k}^+ (z\lambda)^{-k}  \label{K11mz}\\
&& \  + \ \bar\epsilon_+  \tK^{-1}  \exp\left( -(q-q^{-1}) \sum_{n=1}^\infty a_{1,n} (z\lambda)^{- n}\right)   \nonumber\ .
\eeqa
Applying $\gamma'_D$ according to (\ref{gamD0})-(\ref{gamD3}), one finds $\gamma'_D\left(\tilde{K}^-(z)_{11}\right)$ is a power series in the elements of ${U'_q}^{Dr,\triangleright,+}$. Proceeding similarly for the other entries,  $\gamma'_D\left(K^-(z)_{ij}\right) \in {U'_q}^{Dr,\triangleright,+}\otimes {\mathbb C}[[z]]$. Also, the entries are reordered using the defining relations for Drinfeld's currents \cite{DF93}. In particular, one introduces (\ref{psi}) and uses:
\beqa
\tx_{\ell+1}^-g(u) = q^{-2\ell} g(u)\tx_{\ell+1}^-\ ,\qquad  \tx_{\ell}^+g(u) = q^{2\ell} g(u)\tx_{\ell}^+\ .\label{exchxg}
\eeqa
Then, using (\ref{Kmu})  one compares (\ref{K}) to $\cal M(u) \gamma'_D\left({\tilde K}^-(qu^2)\right)   {\cal M}(u)$. 
This gives  (\ref{im1})-(\ref{im4}). 

The second part of the proof concerns the identification of the elements $\gamma'_n$ such that (\ref{condqdet}) holds, i.e. (\ref{Zveecond}). 
In the r.h.s. of (\ref{Zvee}), insert the explicit expressions previously obtained to get the image of 
${\cal C}(u)$ in ${U'_q}^{Dr,\triangleright,+}$. By \cite[Corollary 3.4, Remark 3.5]{Bas20} (recall the substitutions (\ref{cor1}), (\ref{cor2})),  $\nu({\cal C}(u))$ is central: it can be reduced to a function of $C^{1/2}$. To determine this function, it is sufficient to extract all terms of $\nu({\cal C}(u))$ that belong to the center. According to the ordering (\ref{pbw1}) and the reduction rules (\ref{gl1})-(\ref{gl4}), one identifies the subset of terms in the images of $\{{\cY}_\pm(u),{\cZ}_\pm(u)\}$ that are relevant. One finds:
\beqa
&&{\cY}_+(u)\mapsto \bar\epsilon_+ \tK^{-1} (qu^2)^{-1} c(u) + \cdots\ ,\qquad\quad
{\cY}_-(uq)\mapsto \bar\epsilon_- q^{-3}\tK (qu^2)^{-1} c(uq) + \cdots \ ,\nonumber\\
&&{\cZ}_+(uq)\mapsto  \frac{\bar\rho}{(q-q^{-1})} (c(uq) - 1)+ \cdots  \ ,\qquad
{\cZ}_-(u)\mapsto  \frac{\bar\rho}{(q-q^{-1})} (c(u) - 1)+ \cdots  \ ,\nonumber
\eeqa
where 
\beqa
c(u) =  \exp\left( -(q-q^{-1}) \sum_{n=1}^\infty \frac{\gamma'_n}{q^n+q^{-n}} (qu^2\lambda)^{- n}\right)\ ,\label{cu}
\eeqa
and the `dots' correspond to terms that will not contribute. After simplifications, one gets the factorized expression:
\beqa
\nu({\cal C}(u)) - \frac{\bar\rho}{(q-q^{-1})} =  \left((q-q^{-1})q^{-2}\bar\epsilon_+\bar\epsilon_- (qu^2)^{-1} - \frac{\bar\rho}{(q-q^{-1})}\right) c(u)c(uq) \ .\label{nuC}
\eeqa     
The condition $\nu({\cal C}(u)) =0$ leads to 
\beqa
\exp\left( -(q-q^{-1}) \sum_{n=1}^\infty (q\lambda)^{-n}\gamma'_n (qu^2)^{- n}\right) = 1- \frac{(q-q^{-1})^2}{\bar\rho q^2}\bar\epsilon_+\bar\epsilon_-(qu^2)^{-1}\ .
\eeqa
Taking the logarithm on both sides, the corresponding formal power series are identified. It yields to (\ref{gamp}). 

Finally, from Theorem \ref{thm:Tp}, Lemmas \ref{lem:1}, \ref{lem:2}, one knows that $U_q^{T,+} \cong U_q^{DJ,+}$. Then, by lemma \ref{isoUqDrDJp} and Definition \ref{defaltsl2qp}, the map $\nu$ is not surjective.
\end{proof}

Recall (\ref{c1}), (\ref{c2}). Identifying the leading terms of the power series, from the proposition above with (\ref{parfixed}) one gets for instance: 
\begin{example}\label{exi1} The image in ${U'_q}^{Dr,\triangleright,+}$ of $U_q^{T,+}$ is such that:
\beqa
y_0^+ &\mapsto& C^{-1}\tK -q^{-1}(q-q^{-1})C^{-1} \tx_1^- \ ,\nonumber\\ 
y_1^+  &\mapsto& \tK^{-1} -q(q-q^{-1}) \tK^{-1} \tx_0^+ \ ,\nonumber\\
{\tilde z}_1^+ &\mapsto& -(q-q^{-1})^2 \left( q^{-1} C^{-3/2}\tho_1 +(q+q^{-1}) C^{-1}\tx_0^+ \right) + (q-q^{-1})C^{-1}\ ,\nonumber \\
z_1^+ &\mapsto&  (q-q^{-1})^2 \left( q C^{-3/2}\tho_1  -(q+q^{-1}) C^{-1}\tx_1^-\tK^{-1} + q(q^2-q^{-2})C^{-1}\tx_1^- \tK^{-1}\tx_0^+\right) + (q-q^{-1})C^{-1}\ .\nonumber
\eeqa
\end{example}
Using (\ref{isoCP}), it is checked that the images of $\{y_i^+\}_{i=0,1}$ match with (\ref{eq:isol3}).

\begin{rem} Alternative expressions for (\ref{im1})-(\ref{im4}) can be written using the commutations relations (\ref{exchxg}) and $\big[\psi(z),\tK^{\pm 1}\big]= \big[\psi(z),g(u)\big]=0$. 
\end{rem}

\begin{rem}\label{rem:root} The image of the equitable generators in terms of Lusztig's root vectors and $\{K_0,K_1\}$ is obtained as follows. According to the definitions of the root vectors $\{E_{n\delta+\alpha_i}, E_{n\delta}|i=0,1\}\in U_q^{DJ,+}$  (or $\{F_{n\delta+\alpha_i}, F_{n\delta}|i=0,1\}\in U_q^{DJ,-}$) given in \cite{Beck,BCP}, one uses the correspondence:  
\beqa
\tx^+_k &=& E_{k\delta + \alpha_1}\ ,\qquad
\tx^-_{k+1} =- C^{-k-1}\tK E_{k\delta + \alpha_0}\ ,\qquad \qquad\tho_{k+1} = C^{-(k+1)/2} E_{(k+1)\delta}\ ,\label{imr1}\\
\tx^-_{-k} &=& F_{k\delta + \alpha_1}\ ,\qquad
\tx^+_{-k-1} = -  F_{k\delta + \alpha_0}\tK^{-1} C^{k+1}\ ,\qquad  \tho_{-k-1} = C^{(k+1)/2} F_{(k+1)\delta}\label{imr2}
\eeqa
for $k\in {\mathbb N}$ and $\tK =K_1$, $C\tK^{-1}=K_0$ in Proposition \ref{map1}, where $\bar k_\pm,\bar\epsilon_\pm,\lambda$, are chosen such that (\ref{eq:isol3}) is recovered at the leading order of the power series.  
\end{rem}

\subsection{The homomorphism $\delta: U_q^{T,+} \rightarrow {U'_q}^{Dr,\triangleright,+} \otimes U_q^{T,+}$} For $U_q^{T,+}$, a comodule algebra structure can be exhibited as follows. Starting from any K-operator satisfying (\ref{RKz}) and following standard arguments \cite{Skly88}, left or right coactions can be constructed using the FRT presentation. Consider the K-operator in the r.h.s. of (\ref{Ktmz}). A new K-operator can be constructed using a dressing procedure \cite{Skly88}, which leads naturally to a left or right coaction map. For instance\footnote{With respect to the ordering $V_{[1]}\otimes V_{[2]}$:
\beqa
((T)_{[\textsf 1]}(T')_{[\textsf 2]}(T'')_{[\textsf 1]})_{ij} =\sum_{k,\ell=1}^2  (T)_{ik}(T'')_{\ell j}\otimes (T')_{k\ell}  \ .
\eeqa
} 
\begin{prop}\label{prop:delta} $U_q^{T,+}$ is a left comodule algebra	 over ${U'_q}^{Dr,\triangleright,+}$ with coaction map $\delta: U_q^{T,+} \rightarrow {U'_q}^{Dr,\triangleright,+} \otimes U_q^{T,+}$ such that
\beqa
\delta(\tilde K^-(z)) = (\gamma'_D \otimes \gamma'_D)\left((L^-(z q^{(1 \otimes c/2)}))_{[\textsf 1]}  (\tilde K^-(zq^{(c/2 \otimes 1)}))_{[\textsf 2]} 
 ( L^{-,0})_{[\textsf 1]}\right)   \ .\label{coactK}
\eeqa
\end{prop}
\begin{proof} By construction, the r.h.s. satisfies  (\ref{RKz}) for the non-symmetric R-matrix (\ref{def:r}). For $U_q^{T,+}$ to be a comodule algebra, we need to check:
\beqa
(\Delta \otimes id) \circ \delta &=& (id \otimes \delta ) \circ \delta \ ,\label{ax1}\\
({\cal E} \otimes id) \circ \delta &\cong& id \ .\label{ax2}
\eeqa
Firstly, consider (\ref{ax1}). Apply the l.h.s. of (\ref{ax1}) to $\tilde K^-(z)$ in (\ref{Ktmz}) and use the Lax operator coproduct rule (\ref{coprodUqgl2}). Compare the result with the r.h.s. of (\ref{ax1}) applied on $\tilde K^-(z)$. Both expressions coincide. Secondly, consider (\ref{ax2}).
Apply the l.h.s of  (\ref{ax2}) to $\tilde K^-(z)$ in  (\ref{Ktmz}) and use the counit rule (\ref{counitUqgl2}). Thus, we conclude that $U_q^{T,+}$ is a left comodule algebra.
\end{proof}

If needed, the image of the equitable generators by $\delta$ can be extracted in a straightforward manner. As an example, for simplicity let us consider a specialization of (\ref{coactK}), namely the left coaction map $\bar \delta: U_q^{T,+} \rightarrow {U'_q}^{Dr,\triangleright,+}/_{C=1} \otimes U_q^{T,+}$. For the symmetric R-matrix (\ref{R}) using (\ref{simil}) and (\ref{coactK}) at $c=0$ (i.e. $C=1$) it yields to: 
\beqa
\bar\delta( K(u)) = (\gamma'_D \otimes 1)\left(\left( \cal M(u)   L^- (qu^2)  \cal M(u)^{-1}  \right)_{[\textsf 1]}  
(K(u))_{[\textsf 2]}
(L^{-,0})_{[\textsf 1]}\right)/_{C_{[\textsf 1]}=1}\ .\label{deltaKmu}
\eeqa
Now, recall the generating functions (\ref{c1}), (\ref{c2}). 
\begin{lem} \label{coprodform} There exists a coaction map $\bar\delta:  \ U_q^{T,+} \rightarrow  {U'_q}^{Dr,\triangleright,+}/_{C=1} \otimes U_q^{T,+}$ such that:
\beqa
\bar\delta(\cY_+(u)) &\mapsto&  ( qu^2)^{-1}q
 \gamma'_D\left(\tk^-_{1}(qu^2)  (\tk^-_{2,0})^{-1} \tf^-(qu^2)\right)/_{C=1} \otimes  \left( \frac{1}{{\bar k}_+(q+q^{-1})}        \cZ_-(u)     + \frac{{\bar k}_-(q+q^{-1})}{(q-q^{-1})}      \right) \nonumber\\
&&  + \  \gamma'_D\left(\tk^-_{1}(qu^2)  (\tk^-_{2,0})^{-1} \right)/_{C=1}  \otimes  \cY_+(u) \ ,\nonumber\\
\qquad \bar\delta(\cY_-(u) )&\mapsto&   q^{-1} \gamma'_D\left(\te^-(qu^2) \tk^-_{1}(qu^2)  (\tk^-_{1,0})^{-1}\right)/_{C=1} \otimes 
 \left( \frac{1}{{\bar k}_-(q+q^{-1})}        \cZ_+(u)     + \frac{{\bar k}_+(q+q^{-1})}{(q-q^{-1})}     \right)\nonumber\\
&&+\ \gamma'_D\left(    \tk^-_2(qu^2)(\tk^-_{1,0})^{-1} + \  q^{-1} \te^-(qu^2) \tk_{1}^-(qu^2) (\tk^-_{1,0})^{-1}\tf^-(qu^2)  \right)/_{C=1} \otimes \cY_-(u)\ ,\nonumber
\eeqa
\beqa
\bar\delta(\cZ_+(u)) &\mapsto&   \gamma'_D\left(\tk^-_{1}(qu^2)  (\tk^-_{1,0})^{-1}\right)/_{C=1} \otimes \cZ_+(u)  
+ \frac{\bar\rho} {q-q^{-1}}\left( \gamma'_D\left(\tk^-_{1}(qu^2)  (\tk^-_{1,0})^{-1}\right) -1 \right)/_{C=1}    \otimes  1\nonumber\\
&&+\  {\bar k}_-(q+q^{-1})  \gamma'_D\left(\tk^-_{1}(qu^2)  (\tk^-_{1,0})^{-1} \tf^-(qu^2) \right)/_{C=1} \otimes \cY_-(u) \ ,\nonumber\\
\bar\delta(\cZ_-(u)) &\mapsto&  \gamma'_D\left( \tk^-_{2}(qu^2)  (\tk^-_{2,0})^{-1} +     q\te^-(qu^2) k_1^-(qu^2)(\tk^-_{2,0})^{-1} \tf^-(qu^2)         \right)/_{C=1}\otimes \cZ_-(u) \nonumber\\
&&+ \frac{\bar\rho} {q-q^{-1}}\gamma'_D\left(     \tk^-_{2}(qu^2)  (\tk^-_{2,0})^{-1} +     q\te^-(qu^2) k_1^-(qu^2)(\tk^-_{2,0})^{-1} \tf^-(qu^2)       -1  \right)/_{C=1}  \otimes  1\nonumber\\
&&+\  {\bar k}_+qu^2(q+q^{-1})\gamma'_D\left( \te^-(qu^2) \tk^-_{1}(qu^2)  (\tk^-_{2,0})^{-1} \right)/_{C=1} \otimes \cY_+(u) \nonumber
 \ .\nonumber
\eeqa
\end{lem}
\begin{proof}
Compute (\ref{deltaKmu}) using (\ref{Lpm}), (\ref{simil}) and  (\ref{K}) .  Compare the entries of the resulting matrix to $\bar\delta(K(u))$  with (\ref{K}). Applying (\ref{gamD0})-(\ref{gamD3}) to (\ref{eq:cuDre}), (\ref{eq:cuDrf}) and (\ref{kpmz}), one finds $\gamma'_D\left(\bar\delta(K(u))_{ij}\right) \in {U'_q}^{Dr,\triangleright,+}/_{C=1}\otimes U_q^{T,+}\otimes {\mathbb C}[[u^{2}]] $.
\end{proof}

Other examples of left and right  coaction maps can be derived along the same lines. Now, expanding the power series on both sides of the above equations using (\ref{c1}), (\ref{c2}), (\ref{eq:cuDre})-(\ref{eq:cuDrk}) with (\ref{kpmz}), one gets the image by $\bar\delta$ of the generators of  $U_q^{T,+}$.  
\begin{example}\label{excop}
\beqa
\bar\delta(y_1^+) &=& -q(q-q^{-1})\tK^{-1} \tx_0^+ \otimes 1 + \tK^{-1} \otimes y_1^+ \ ,\nonumber\\
\bar\delta(y_0^+) &=& -q^{-1}(q-q^{-1})\tx_1^- \otimes 1 + \tK \otimes y_0^+ \ .\nonumber
\eeqa
\end{example}
Using (\ref{isoCP}) at $C=1$, for (\ref{eq:isol3}) one finds $\bar\delta$ coincides with $\Delta$, see (\ref{coprod}).\vspace{1mm}

Note that tensor product representations for $U_q^{T,+}$ can be obtained from \cite[Section 4]{Bas20}, adapting the definitions of the generators and conventions. See \cite[Proposition 4.5]{Bas20}.

\section{The central extension of $U_q^{T,+}$}\label{secextU}
In this section, following \cite{Ter19b} the central extension of ${U}_q^{T,+}$ and its center are considered. Below, they are denoted respectively ${\cal U}_q^{T,+}$ and ${\textsf C}^+$. Specializing the results of \cite{Bas20}, a Freidel-Maillet type presentation for ${\cal U}_q^{T,+}$ is given. Then, following \cite{Bas20} the alternating subalgebras $U_q(\widehat{gl_2})^{ \triangleright,+}$, ${U'_q}(\widehat{gl_2})^{ \triangleright,+}$ and center ${\textsf C}^\triangleright$ are introduced. By analogy with the analysis in previous section, the Freidel-Maillet type presentation is used to compute the images of the generators of ${\cal U}_q^{T,+}$ and ${\textsf C}^+$ in ${U'_q}(\widehat{gl_2})^{ \triangleright,+}$ and ${\textsf C}^\triangleright$, respectively. \vspace{1mm}

The following definitions are straightforward adaptations of \cite{Ter19b}.
\begin{defn}\label{ext1} ${\cal U}_q^{T,+}$  is the unital associative $\mathbb C(q)$-algebra with {\it equitable} generators $\{Y_{-k}^+,Y_{k+1}^+,Z_{k+1}^+,{\tilde Z}_{k+1}^+|k\in {\mathbb N}\}$ subject to the relations (\ref{def1})-(\ref{def11}) with the substitutions:
\beqa
y_{-k}^+ \rightarrow Y_{-k}^+\ ,\quad y_{k+1}^+ \rightarrow Y_{k+1}^+\ ,\quad z_{k+1}^+ \rightarrow Z_{k+1}^+\ ,\quad {\tilde z}_{k+1}^+ \rightarrow {\tilde Z}_{k+1}^+\ .
\eeqa
\end{defn}
\begin{defn}\label{ext2} The center ${\textsf C}^+$ is the subalgebra of ${\cal U}_q^{T,+}$ generated by the elements:
\beqa
{\textsf C}_{n+1}=(q^2-q^{-2}) \sum_{k=0}^{n} q^{-2n+2k-1} Y_{k+1}^+Y_{-n+k}^+ - \frac{(q-q^{-1})}{\bar\rho}\sum_{k=0}^{n+1}q^{2k-2n-2} Z_k^+\tilde{Z}_{n+1-k}^+ \ ,\quad n \in {\mathbb N}\ \label{Cnp1}
\eeqa
with $Z^+_0=\tilde{Z}^+_0=\bar\rho/(q-q^{-1})$.
\end{defn}

Note that the automorphisms $\sigma,S$, given by (\ref{sig}), (\ref{autS}), naturally extend from $U_q^{T,+}$ to ${\cal U}_q^{T,+}$; See \cite[Section 8]{Ter19b} for details. Importantly, using the defining relations of ${\cal U}_q^{T,+}$ one shows that the central elements ${\textsf C}_{n+1}$ are fixed under the action of $\sigma,S$  \cite[Proposition 8.3]{Ter19b}.\vspace{1mm} 

A Freidel-Maillet type presentation for ${\cal U}_q^{T,+}$ follows from \cite[Theorem 3.1]{Bas20}, adapting the notations. We refer the reader to this work for the proof of the Theorem below. Introduce the generating functions:
\begin{align}
{\tY}_+(u)=\sum_{k\in {\mathbb N}}{Y}^+_{k+1}U^{-k-1} \ , \quad {\tY}_-(u)=\sum_{k\in  {\mathbb N}}{Y}^+_{-k}U^{-k-1} \ ,\label{ct1}\\
 \quad {\tZ}_+(u)=\sum_{k\in {\mathbb N}}{\tilde Z}^+_{k+1}U^{-k-1} \ , \quad {\tZ}_-(u)=\sum_{k\in {\mathbb N}}{Z}^+_{k+1}U^{-k-1} \ .\label{ct2}
\end{align}
\begin{thm}\label{thm3} ${\cal U}_q^{T,+}$ has a presentation of Freidel-Maillet type.
 Let $\tK(u)$ be a square matrix such that 
\beqa
&&  \tK(u)=
       \begin{pmatrix}
      uq \tY_+(u) &  \frac{1}{{\bar k}_-(q+q^{-1})}\tZ_+(u) + \frac{{\bar k}_+(q+q^{-1})}{(q-q^{-1})} \\
     \frac{1}{{\bar k}_+(q+q^{-1})}\tZ_-(u) + \frac{{\bar k}_-(q+q^{-1})}{(q-q^{-1})}  & uq \tY_-(u) 
      \end{pmatrix} \label{tK}
\eeqa
with (\ref{ct1})-(\ref{ct2}).
The defining relations are given by:
\begin{align} R(u/v)\ (\tK(u)\otimes {\mathbb I})\ R^{(0)}\ ({\mathbb I} \otimes \tK(v))\
= \ ({\mathbb I} \otimes \tK(v))\  R^{(0)}\ (\tK(u)\otimes {\mathbb I})\ R(u/v)\ .
\label{REt} \end{align}
\end{thm}

In the Freidel-Maillet framework, a generating function for central elements of ${\textsf C}^+$ is derived from the quantum determinant of the K-operator (\ref{tK}). Introduce the generating function with coefficients (\ref{Cnp1}) 
\beqa
{\textsf C}(u) =  \sum_{k=0}^\infty {\textsf C}_{k+1} U^{-k-1}\ .
\eeqa
By previous comments, note that $\sigma({\textsf C}(u))={\textsf C}(u)$. The following proposition is an alternative to \cite[Section 13]{Ter19b}, adapted to the equitable case. 
\begin{prop} \label{propqdet}  The quantum determinant
\beqa
\Gamma(u)=\tr_{12}\big(P^{-}_{12}(\tK(u)\otimes I\!\!I)\ R^{(0)} (I\!\!I \otimes \tK(uq))\big) \ \ \label{tgamma}
\eeqa
generates  ${\textsf C}^+$. 
\end{prop}
\begin{proof} Firstly, one shows that $\Gamma(u)$ is central i.e. $\big[\Gamma(u),(\tK(u))_{ij}\big]=0$. We refer the reader to \cite[Proposition 3.3]{Bas20} for details. Secondly, inserting (\ref{tK}) into the r.h.s. of (\ref{tgamma}), one gets
\beqa
\Gamma(u)= \frac{1}{(q-q^{-1})}\left( {\textsf C}(u)  -\frac{\bar\rho}{(q-q^{-1})}\right)\ \label{GC}
\eeqa
where
\beqa
{\textsf C}(u) &=& (q-q^{-1})u^2q^2\tY_+(u)\tY_-(uq)  - \frac{(q-q^{-1})}{\bar\rho} \tZ_-(u)\tZ_+(uq)  - \tZ_-(u) - \tZ_+(uq)\ .\label{tZvee}
\eeqa
\end{proof}

Also, the analogs of \cite[Lemma 3.3, Lemma 2.8]{Ter19b}  take the following form. Recall Theorems \ref{thm1}, \ref{thm3}.
\begin{lem}
There exists a surjective homomorphism ${\cal U}_q^{T,+} \rightarrow U_q^{T,+}$ that sends
\beqa
\tK(u) \mapsto K(u) \ ,\qquad \Gamma(u) \mapsto -\frac{\bar\rho}{(q-q^{-1})^2}\ .
\eeqa
\end{lem}

An embedding of ${\cal U}_q^{T,+}$ into a subalgebra of $U_q(\widehat{gl_2})$  can be obtained using the the FRT presentation  of Theorem \ref{def:UqRS}.  To prepare the analysis below, some results from \cite{Bas20} are needed. We refer to \cite[Section 5]{Bas20} for details and references. The definition below is a variation of \cite[Definition 5.12, eqs. (5.52)-(5.53)]{Bas20}. 
\begin{defn}
\label{defaltq}
\beqa
U_q(\widehat{gl_2})^{\triangleright,\pm} &=& \{C^{\mp k/2}\tK^{-1}\tx_{k}^\pm, C^{\pm (k+1)/2}\tx^\mp_{k+1}, a_{1,k+1},a_{2,k+1}|k\in {\mathbb N}\} \ ,\nonumber\\
U_q(\widehat{gl_2})^{\triangleleft,\pm}&=& \{C^{\mp k/2}\tx_{-k}^\pm, C^{\pm (k+1)/2} \tx^\mp_{-k-1}\tK, a_{1,-k-1},a_{2,-k-1}| k\in {\mathbb N}\} \ . \nonumber
\eeqa
We call $U_q(\widehat{gl_2})^{\triangleright,\pm}$ and $U_q(\widehat{gl_2})^{ \triangleleft,\pm}$ the right and left 
alternating subalgebras of  $U_q(\widehat{gl_2})$.  The subalgebra generated by $\{\tK^{\pm 1},C^{\pm 1/2}\}$ is denoted  $U_q(\widehat{gl_2})^{\diamond}$.
\end{defn}

Importantly, it is known that the elements 
\beqa
\gamma_m = q^{m}a_{1,m} + q^{-m}a_{2,m} \quad \mbox{for}\quad m \in {\mathbb Z}^*\ \label{decg}
\eeqa
generate the center ${\cal C}$ of $U_q(\widehat{gl_2})$ (see e.g. \cite{FMu}). It follows:
\begin{defn}\label{Ct} The center $\cal C^\triangleright$ (resp. $\cal C^\triangleleft$) of  $U_q(\widehat{gl_2})^{\triangleright,\pm}$ (resp.  $U_q(\widehat{gl_2})^{ \triangleleft,\pm}$) is generated by  
 $\gamma_m$ (resp. 
 $\gamma_{-m}$) with $m\in {\mathbb N}^*$.
\end{defn}

Extensions of the alternating subalgebras of $U_q(\widehat{gl_2})$ are now introduced.
\begin{defn} ${U'_q}(\widehat{gl_2})^{\triangleright,\pm}$ (resp. ${U'_q}(\widehat{gl_2})^{\triangleleft,\pm}$) denote the subalgebras of $U_q(\widehat{gl_2})$ generated by ${U}_q(\widehat{gl_2})^{\triangleright,\pm}$ (resp. ${U}_q(\widehat{gl_2})^{\triangleleft,\pm}$) and $\{\tK^{\pm 1},C^{\pm 1/2}\}$.
\end{defn}

For the quantum algebra $U_q(\widehat{gl_2})$, it is known that $U_q(\widehat{gl_2}) \cong U_q^{Dr}\otimes {\cal C}$. Thus, for the alternating subalgebras analog properties hold. Recall Definition \ref{defaltsl2qp} and (\ref{a2mbis}). One has:
\beqa
{U'_q}(\widehat{gl_2})^{\triangleright,\pm} \cong {U'_q}^{Dr,\triangleright,\pm}  \otimes \cal C^\triangleright\ ,\qquad {U'_q}(\widehat{gl_2})^{\triangleleft,\pm} \cong {U'_q}^{Dr,\triangleleft,\pm} \otimes \cal C^\triangleleft \ . \nonumber
\eeqa

The embedding of the Freidel-Maillet algebra (\ref{REt}) into the FRT presentation of $U_q(\widehat{gl_2})$ is now studied. Generalizing the results of previous section, the map (\ref{Kmu}) allows to establish the precise relation between the equitable generators $ \{{Y}^+_{-k}, {Y}^+_{k+1},  {Z}^+_{k+1}, {\tilde Z}^+_{k+1}\}$ and the generators of alternating subalgebras of $U_q(\widehat{gl_2})$.
Recall (\ref{cu}) and Proposition \ref{map1}. Introduce the generating function in the central elements: 
\beqa
{\textsf c}(u) = \exp\left( -(q-q^{-1}) \sum_{n=1}^\infty \frac{\gamma_n}{q^n+q^{-n}} (qu^2\lambda)^{- n}\right)\  \in {\cal C}^\triangleright \otimes {\mathbb C}[[u^2]].\label{tcu}
\eeqa
\begin{prop}\label{prop:YZmap} There exists an injective homomorphism $\mu: \ {\cal U}_q^{T,+} \rightarrow {U'_q}(\widehat{gl_2})^{\triangleright,+}$ such that
\beqa
\tY_\pm(u) &\mapsto&  \nu(\cY_\pm(u)) c(u)^{-1} {\textsf c}(u)\ ,\label{tYmap}\\
\tZ_\pm(u) &\mapsto&  \left(\nu(\cZ_\pm(u)) + \frac{\bar\rho}{(q-q^{-1})}\right)c(u)^{-1}{\textsf c}(u) - \frac{\bar\rho}{(q-q^{-1})}\ .\label{tZmap}
\eeqa
\end{prop}
\begin{proof}
One expands explicitly (\ref{Ktmz}) using (\ref{Lpm}). For the entry $({\tilde K}^-(z))_{11}$, we previously obtained (\ref{K11mz}). Inserting  (\ref{a2mbis}) and using (\ref{gamD2}), it factorizes as:
\beqa
(\tilde K^-(z))_{11} &=& \exp\left( -(q-q^{-1}) \sum_{n=1}^\infty \frac{(\gamma_{n}-\gamma'_{n})}{q^n+q^{-n}} (z\lambda)^{- n}\right)   \gamma_D'((\tilde K^-(z))_{11}) 
\ .
\eeqa
Actually, other entries $(\tilde K^-(z))_{ij} \in {U'_q}(\widehat{gl_2})^{\triangleright,\pm} \otimes {\mathbb C}[[z]]$ and similarly factorize in terms of  $\gamma_D'((\tilde K^-(z))_{ij})$. Thus, a solution of (\ref{REt}) is given by the  K-operator:
\beqa
\cal M(u) \left({\tilde K}^-(qu^2)\right) {\cal M}(u)  =  {\textsf c}(u) c(u)^{-1}\cal M(u) \gamma'_D\left({\tilde K}^-(qu^2)\right)   {\cal M}(u) \label{tKgl}
\eeqa
with (\ref{cu}), (\ref{tcu}). Then,  one compares (\ref{K}) to (\ref{tKgl}) which gives  (\ref{tYmap})-(\ref{tZmap}).
\end{proof}

In terms of central elements of $U_q(\widehat{gl_2})$, the quantum determinant takes a rather simple form.
\begin{cor}\label{cor:Cmap}
\beqa
\Gamma(u) \stackrel{\mu}{\mapsto}  \left(  \frac{\bar\epsilon_+\bar\epsilon_-(q-q^{-1})^2 - \bar k_+\bar k_-(q+q^{-1})^2}{q^3u^2(q-q^{-1})^2}\right)\exp\left( -(q-q^{-1}) \sum_{n=1}^\infty \gamma_n (q^2u^2\lambda)^{- n}\right)\ \label{Cmap}
\eeqa
with (\ref{parfixed}).
\end{cor}
\begin{proof}
Applying $\mu$ to (\ref{tZvee}) and using (\ref{tYmap}), (\ref{tZmap}), one gets:
\beqa
\mu({\textsf C}(u)) -\frac{\bar\rho}{(q-q^{-1})} = {\textsf c}(u){\textsf c}(uq)c(u)^{-1}c(uq)^{-1}\left(\nu({\cal C}(u)) -\frac{\bar\rho}{(q-q^{-1})}\right)\ .
\eeqa 
Using (\ref{rho}), (\ref{nuC}), (\ref{GC}), (\ref{tcu}), the r.h.s. of (\ref{Cmap}) follows.
\end{proof}

\begin{rem} An alternative expression for  $\mu(\Gamma(u))$ is derived as follows. Inserting (\ref{decg}) into (\ref{tcu}) and using the second eq. of (\ref{a2mbis}), (\ref{hh}) and (\ref{psi}), one gets:
\beqa
\Gamma(u) \stackrel{\mu}{\mapsto}   \left(  \frac{\bar\epsilon_+\bar\epsilon_-(q-q^{-1})^2 - \bar k_+\bar k_-(q+q^{-1})^2}{q^3u^2(q-q^{-1})^2}\right) {\textsf g}(u) \psi(q^2u^2\lambda){\textsf g}(uq)\ \label{Cmapbis}
\eeqa
with 
\beqa
 {\textsf g}(u) = \exp\left( -(q-q^{-1}) \sum_{n=1}^\infty a_{1,n} (qu^2\lambda)^{-n}\right) \ .\label{Cmapbisalt}
\eeqa
\end{rem}
\vspace{1mm}

\section{Concluding remarks}
The results of this paper together with \cite{Bas20} can be summarized as follows. In \cite{Bas20}, it was shown that the alternating presentation for an algebra $U_q^+$ and its central extension ${\cal U}_q^+$ introduced and studied in \cite{Ter18,Ter19,Ter19b} 
admits a presentation of Freidel-Maillet type. For ${\cal U}_q^+$, this presentation consists of a K-operator satisfying (\ref{REt}), which entries are generating functions in the alternating generators of ${\cal U}_q^+$ \cite{Ter19,Ter19b}; See \cite[Theorem 3.1]{Bas20}.
To get the analog presentation for $U_q^+$, a condition for the quantum determinant of the K-operator is asserted. It reads (\ref{condqdet}). Now, by Lemmas \ref{lem:1}, \ref{lem:2}, two different embeddings of the alternating presentation of $U_q^+$ into 
 $U_q(\widehat{sl_2})$ can be considered
: \vspace{1mm}

(i) The Drinfeld-Jimbo (or Chevalley) type:  the alternating presentation for $U_q^{DJ,+}$ - the standard positive part of $U_q(\widehat{sl_2})$; \vspace{1mm}

(ii) The equitable (or Ito-Terwilliger) type: the alternating presentation for $U_q^{IT,+}$ - the non-standard positive part of $U_q(\widehat{sl_2})$.  \vspace{1mm}

In this paper, the alternating presentation for the equitable type is denoted $U_q^{T,+}$, and its central extension ${\cal U}_q^{T,+}$. Its generators are called the equitable generators to avoid any confusion with the Drinfeld-Jimbo type. \vspace{1mm}

The alternating presentations (i),(ii), are studied in details in \cite{Bas20,Ter21b} and the present paper. The approach followed  in \cite{Bas20} and here is based on Freidel-Maillet type presentations. In this framework, to the K-operator of $U_q^+$ (see \cite[eq. (3.8)]{Bas20} one associates a K-operator of Drinfeld-Jimbo type for (i), or of equitable type (\ref{K}) for (ii). For each type, embeddings into Yang-Baxter subalgebras of $U_q^{RS}$ \cite{RS,DF93} and Drinfeld second realization $U_q^{Dr}$ \cite{Dr} are studied in details. They are characterized explicitly using L-operators and Drinfeld generators as follows: \vspace{1mm}

(i') The K-operator of Drinfeld-Jimbo type is the image of (\ref{Kmu}) by $\gamma'_D$  for $\bar\epsilon_\pm=0,\bar k_+=q^2,\bar k_-=-q^{-1}$ and $\gamma'_n=0$, $\forall n$. As a corollary, the image of the alternating generators in $U_q^{Dr}$ is obtained in terms of Drinfeld generators/root vectors from (\ref{im1})-(\ref{im4}); See \cite[Propositions 5.27]{Bas20};\vspace{1mm}

(ii') The K-operator of equitable type is the image of (\ref{Kmu}) by $\gamma'_D$ for (\ref{parfixed}), (\ref{gamp}). As a corollary,  the image of the equitable generators  in terms of Drinfeld generators is obtained. It follows from (\ref{im1})-(\ref{im4}), thus generalizing (\ref{eq:isol3}) of \cite{IT03}. See Remark \ref{rem:root} for the corresponding expressions in terms of root vectors.\vspace{1mm}

Analogous results hold for their central extensions related with subalgebras of $U_q(\widehat{gl_2})$, see \cite[Section 5.2]{Bas20}) and Section \ref{secextU} of this paper. Furthermore, for the central extension of any type (i) and (ii) the image of the quantum determinant (\ref{Cmap}) coincides, up to an overall factor, with a generating function for `half' of central elements (\ref{decg}) of $U_q(\widehat{gl_2})$.\vspace{1mm}

The results of \cite{Bas20} and the present paper show that Freidel-Maillet type algebras provide a unified framework for $U_q^{DJ,+}$ and $U_q^{IT,+}$.  Actually, this unified framework can be extended to $U_q(\widehat{sl_2})$ (see \cite[Section 6]{Bas20} for the Drinfeld-Jimbo type), thus providing an alternative to the FRT presentation for $U_q(\widehat{sl_2})$ \cite{DF93}. For the Drinfeld-Jimbo type, the existence of a Freidel-Maillet type presentation can be understood from the FRT presentation using a Drinfeld twist. The most interesting case is the Freidel-Maillet presentation for $U_q(\widehat{sl_2})$ of equitable type. Details will be considered elsewhere. 
As a preliminary, for the case of $U_q(sl_2)$ the Freidel-Maillet type presentation unifying the Drinfeld-Jimbo and equitable presentations is introduced and studied in \cite{Bas21}. \vspace{1mm}

Let us also mention that FRT presentations for higher rank affine Lie algebras have been recently achieved, see \cite{JLM,JLMBD,LP21}. By analogy, Freidel-Maillet and alternating presentations of Drinfeld-Jimbo or equitable type for higher rank cases are expected.\vspace{1mm}   

From the perspective of physics, it seems natural to study further Freidel-Maillet type algebras of Drinfeld-Jimbo or equitable type. Indeed, it is known that FRT presentations of quantum algebras provide a powerful framework for the explicit construction and analysis of quantum integrable models such as spin chains, using the Bethe ansatz or q-vertex operators's techniques for instance. By analogy, it would be natural to investigate the class of quantum integrable models generated from Freidel-Maillet type presentations. 
\vspace{1cm}

\noindent{\bf Acknowledgments:} I thank Paul Terwilliger for many discussions, kind explanations of his work and important comments on the manuscript. Also, I thank him for sharing the unpublished results of \cite{Ter21b} which motivated the analysis of Section \ref{secextU}. P.B.  is supported by C.N.R.S. 
\vspace{0.2cm}

\begin{appendix}

\section{Drinfeld-Jimbo  and Drinfeld (second realization) presentation of $U_q(\widehat{sl_2})$}\label{ap:A}
\vspace{2mm}
For the quantum affine Kac-Moody algebra $U_q(\widehat{sl_2})$, two standard presentations are recalled. The {\it Drinfeld-Jimbo} presentation  $U_q^{DJ}$ and the  {\it Drinfeld (second) presentation}  $U_q^{Dr}$, see e.g. \cite[p.392]{CPb}. 
\subsection{Drinfeld-Jimbo presentation $U_q^{DJ}$}
Define the extended Cartan matrix $\{a_{ij}\}$ ($a_{ii}=2$,\ $a_{ij}=-2$ for $i\neq j$). The quantum affine algebra $U_{q}(\widehat{sl_2})$ over ${\mathbb C}(q)$ is generated by  $\{E_j,F_j,K_j^{\pm 1}\}$, $j\in \{0,1\}$ which satisfy the defining relations
\beqa 
K_iK_j=K_jK_i\ , \quad K_iK_i^{-1}=K_i^{-1}K_i=1\ , \quad
K_iE_jK_i^{-1}= q^{a_{ij}}E_j\ ,\quad
K_iF_jK_i^{-1}= q^{-a_{ij}}F_j\ ,\quad
[E_i,F_j]=\delta_{ij}\frac{K_i-K_i^{-1}}{q-q^{-1}}\
\nonumber\eeqa
together with the $q-$Serre relations  ($i\neq j$)
\beqa
 \big[E_i, \big[E_i, \big[E_i,E_j \big]_{q} \big]_{q^{-1}} \big]&=&0\ ,\label{defUqDJp}\\
 \big[F_i, \big[F_i, \big[F_i,F_j \big]_{q} \big]_{q^{-1}} \big]&=& 0\ . \label{defUqDJm}
\eeqa
The  product $C=K_0K_1$ is the central element of the algebra. 

The Hopf algebra structure is ensured by the existence of a
comultiplication $\Delta$
, antipode ${\cal S}$
and a counit ${\cal E}$
with
\beqa \Delta(E_i)&=& 1 \otimes E_i + E_i \otimes K_i
\ ,\label{coprod} \\
 \Delta(F_i)&=&F_i \otimes 1 +   K_i^{-1}\otimes F_i\ ,\nonumber\\
 \Delta(K_i)&=&K_i\otimes K_i\ ,\nonumber
\eeqa
%
%
%
\beqa {\cal S}(E_i)=-E_iK_i^{-1}\ ,\quad {\cal S}(F_i)=-K_iF_i\ ,\quad {\cal S}(K_i)=K_i^{-1} \qquad {\cal S}({1})=1\
\label{antipode}\nonumber\eeqa
and\vspace{-0.3cm}
\beqa {\cal E}(E_i)={\cal E}(F_i)=0\ ,\quad {\cal
E}(K_i)=1\ ,\qquad {\cal E}(1)=1\
.\label{counit}\nonumber\eeqa
%
%
%

%
\subsection{Drinfeld's second realization $U_q^{Dr}$}
A second presentation for the quantum affine algebra $U_q(\widehat{sl_2})$, known as the {\it Drinfeld's second realization}, is now recalled.  In \cite{Dr}, it is shown that    $U_q(\widehat{sl_2})$ is isomorphic to the associative algebra over ${\mathbb C}(q)$ with generators $\{{\tx}_k^{\pm}, \tho_{\ell},
\tK^{\pm 1} |k\in {\mathbb Z},\ell\in {\mathbb Z}\backslash
\{0\} \}$, central elements $C^{\pm 1/2}$ and the following relations (see e.g. \cite[Theorem 12.2.1]{CPb}):  
\beqa
&&  C^{1/2}C^{-1/2}=1\ ,\quad \tK\tK^{-1}=\tK^{-1}\tK=1 \ ,\label{gl1}\\
&&\big[ \tho_k,\tho_\ell   \big]  =  \delta_{k+\ell,0}\frac{1}{k}\big[ 2k \big]_q \frac{C^k-C^{-k}}{q-q^{-1}} \ , \qquad \big[\tho_k, \tK^{\pm 1}  \big]  =0\ , \label{hh} \\
&& \big[ \tho_k, \tx^\pm_\ell \big]  = \pm \frac{1}{k}\big[ 2k\big]_q C^{\mp |k|/2} \tx^\pm_{k+\ell}\ ,\\ \label{hx}
&& \tK \tx^\pm_k \tK^{-1}    =  q^{\pm 2} \tx^\pm_{k} \ ,\label{gl2}\\
&& \tx^\pm_{k+1} \tx^\pm_\ell  -q^{\pm 2} \tx^\pm_{\ell} \tx^\pm_{k+1} = q^{\pm 2} \tx^\pm_{k} \tx^\pm_{\ell+1} - \tx^\pm_{\ell+1} \tx^\pm_k \ ,\label{gl3}\\     
&& \big[ \tx^+_k, \tx^-_\ell   \big]  = \frac{(C^{(k-\ell)/2} \psi_{k+\ell} - C^{-(k-\ell)/2}\phi_{k+\ell})}{q-q^{-1}}\ ,\label{gl4}
\eeqa
where the $\psi_{k}$ and $\phi_{k}$  are defined by the following equalities of formal power series in the indeterminate $z$:
\beqa
\psi(z)&=&\sum_{k=0}^\infty \psi_{k} z^{-k} = \tK \exp\left( (q-q^{-1}) \sum_{k=1}^{\infty} \tho_k  z^{-k} \right) \ ,\label{psi}\\
 \phi(z)&=& \sum_{k=0}^\infty \phi_{-k} z = \tK^{-1} \exp\left(- (q-q^{-1}) \sum_{k=1}^{\infty} \tho_{-k}  z \right)\ .\label{phi}
\eeqa
Note that there exists an automorphism    such that:
\beqa
\theta: &&   \tx^\pm_k \mapsto  \tx^\mp_k \ ,\quad \tho_k \mapsto -\tho_k\ ,\quad \tK \mapsto \tK\ ,\quad C \mapsto C^{-1}, \ \quad q \mapsto q^{-1}\ .\label{thetq}
\eeqa

An isomorphism $U_q^{DJ}\rightarrow U_q^{Dr}$ is given by (see e.g \cite[p. 393]{CPb}:
\beqa
\qquad
 K_0 \mapsto C\tK^{-1}\ , \quad K_1 \mapsto \tK \ ,\quad
 E_1 \mapsto  \tx_0^+\ , \quad E_0 \mapsto  \tx_1^-\tK^{-1}\ ,\quad
 F_1 \mapsto \tx_0^-\ , \quad F_0 \mapsto \tK\tx_{-1}^+\ .\label{isoCP}
\eeqa

Note that it is still an open problem to find the complete Hopf algebra isomorphism between $U_q^{DJ}$ and $U_q^{Dr}$. Only partial information is known, see e.g. \cite[Section 4.4]{CP}.

\end{appendix}

\vspace{2mm}

\providecommand{\bysame}{\leavevmode\hbox to3em{\hrulefill}\thinspace}

\end{document}